\documentclass{amsart}
\usepackage[utf8]{inputenc}
\usepackage{amsmath}
\usepackage{amsfonts}
\usepackage{amssymb}
\usepackage{amsthm}
\usepackage{amscd}
\usepackage{float}
\usepackage{tikz}
\usepackage{graphicx}
\usepackage[colorlinks=true]{hyperref}
\hypersetup{urlcolor=blue, citecolor=red}
\usepackage{hyperref}

\newtheorem{theorem}{Theorem}[section]
\newtheorem{corollary}[theorem]{Corollary}

\newtheorem{lemma}[theorem]{Lemma}
\newtheorem{proposition}[theorem]{Proposition}

\theoremstyle{definition}

\newtheorem{remark}[theorem]{Remark}

\newcommand{\R}{\mathbb{R}}
\newcommand{\N}{\mathbb{N}}
\newcommand{\C}{\mathbb{C}}
\newcommand{\T}{\mathbb{T}}
\newcommand{\Z}{\mathbb{Z}}

\def\Xint#1{\mathchoice
{\XXint\displaystyle\textstyle{#1}}%
{\XXint\textstyle\scriptstyle{#1}}%
{\XXint\scriptstyle\scriptscriptstyle{#1}}%
{\XXint\scriptscriptstyle\scriptscriptstyle{#1}}%
\!\int}
\def\XXint#1#2#3{{\setbox0=\hbox{$#1{#2#3}{\int}$ }
\vcenter{\hbox{$#2#3$ }}\kern-.6\wd0}}

\def\dashint{\Xint-}

\begin{document}

\title[Refinements of Strichartz estimates]{Refinements of Strichartz estimates on tori and applications}
\author{Robert Schippa}
\address{Korea Institute of Advanced Study, Hoegi-ro 85, Dongdaemun-gu, 02455 Seoul, Republic of Korea}
\email{rschippa@kias.re.kr}
\begin{abstract}
We show trilinear Strichartz estimates in one and two dimensions on frequency-dependent time intervals. These improve on the corresponding linear estimates of periodic solutions to the Schr\"odinger equation. The proof combines decoupling iterations with bilinear short-time Strichartz estimates. 
 Secondly, we use decoupling to show new linear Strichartz estimates on frequency dependent time intervals. We apply these in case of the Airy propagator to obtain the sharp Sobolev regularity for the existence of solutions to the modified Korteweg-de Vries equation.
\end{abstract}

\maketitle

\section{Introduction}

Discrete Fourier restriction deals with estimates of the kind
\begin{equation}
\label{eq:DiscreteFourierRestriction}
\| \sum_{\substack{ k \in \Z^d, \\ |k| \leq N}} a_k e^{i(x \cdot k - t k^2)} \|_{L_{t,x}^p([0,1] \times \T^d)} \lesssim K_p(N) \| a_k \|_{\ell_k^2}.
\end{equation}

After pioneering work by Zygmund \cite{Zygmund1974}, Bourgain \cite{Bourgain1993A} studied discrete Fourier restriction to obtain Strichartz estimates for the periodic Schr\"odinger propagator:
\begin{equation}
\label{eq:DiscreteStrichartz}
\| e^{it \Delta} P_N f \|_{L^p_{t,x}([0,1] \times \T^d)} \lesssim K_p(N) \| f \|_{L^2},
\end{equation}
where $P_N$ denotes frequency projection to frequencies of size $N$. A wide range of periodic Strichartz estimates with constant $K_p(N)$ sharp up to endpoints were proved in \cite{Bourgain1993A,Bourgain1993B}. More recently, Bourgain--Demeter \cite{BourgainDemeter2015} showed sharp $\ell^2$-decoupling inequalities. The $\ell^2$-decoupling inequalities imply the Strichartz estimates on $\T^d$ by continuous approximation; see below. 

 Let $\mathcal{E}$ denote the Fourier extension operator of the truncated paraboloid:
\begin{equation*}
\mathcal{E} f(x,t) = \int_{\{ \xi \in \R^d, \; |\xi| \leq 1 \}} e^{i(x \cdot \xi - t|\xi|^2)} f(\xi) d\xi, \quad (x,t) \in \R^d \times \R.
\end{equation*}
The following estimates were proved in \cite{BourgainDemeter2015}:
\begin{equation}
\label{eq:l2DecouplingIntroduction}
\| \mathcal{E} f \|_{L^p_{t,x}(B_{d+1}(0,R))} \lesssim_\varepsilon R^\varepsilon \big( \sum_{\theta: R^{-\frac{1}{2}}\text{-ball}} \| \mathcal{E} f_{\theta} \|^2_{L^p_{t,x}(w_{B_{d+1}(0,R)})} \big)^{\frac{1}{2}}
\end{equation}
with $R \gg 1$, $2 \leq p \leq \frac{2(d+2)}{d}$. $w_{B(0,R)}$ denotes a weight decaying off $B_{d+1}(0,R)$ polynomially (say with degree $100d$). $\{ \theta \}$ denotes an essentially disjoint cover of $B(0,1)$ with $R^{-\frac{1}{2}}$-balls, and we denote
\begin{equation*}
\mathcal{E} f_{\theta} = \int_{\theta} e^{i(x \cdot \xi - t|\xi|^2)} f(\xi) d\xi.
\end{equation*}


\medskip

 To solve nonlinear equations, multilinear Strichartz estimates which improve on the corresponding linear estimate in certain cases turned out to be crucial.

Consider the multilinear estimates
\begin{equation}
\label{eq:Multilinear}
\| \prod_{i=1}^n P_{N_i} e^{it \Delta} f_i \|_{L^2_{t,x}([0,T] \times \T^d)} \lesssim C(T,N_1,\ldots,N_n) \prod_{i=1}^n \| f_i \|_{L^2(\T^d)}
\end{equation}
for frequencies $N_1 \gtrsim \ldots \gtrsim N_n$. It became clear in Bourgain's work \cite{Bourgain2004} that further refinements of \eqref{eq:Multilinear} for frequency-dependent times $T=T(N_1,\ldots,N_n)$ can be used in the analysis of solutions to mass-critical Schr\"odinger equations with infinite energy. Precisely, in \cite{Bourgain2004} the following qualitative trilinear Strichartz estimate was proved:

\begin{theorem}[{\cite[Section~4]{Bourgain2004}}]
\label{thm:BourgainImprovedStrichartz}
Fix large numbers $N_j \in 2^{\N}$ ($1 \leq j \leq 3$) and $D$, which satisfy $N_1 \geq N_2 \geq N_3$ and for some $1 \geq \beta > 0$ 
\begin{equation}
\label{eq:SizeConstraint1}
N_1 > N_3^{1+ \beta},
\end{equation}
and
\begin{equation}
\label{eq:SizeConstraint2}
D > N_1^\beta.
\end{equation}
Assume
\begin{equation}
\label{eq:FourierSupportConstraint}
\text{supp} (\hat{f}_j) \subseteq [N_j,2N_j] \qquad (1 \leq j \leq 3).
\end{equation}
Then
\begin{equation}
\label{eq:ImprovedTrilinearStrichartzBourgain}
\int_{[0,1/D] \times \T} \prod_{j=1}^3 |e^{it \Delta} f_j |^2 dx dt \leq N_1^{-\delta'} \prod_{j=1}^3 \| f_j \|_2^2,
\end{equation}
where $\delta' = \delta'(\beta) > 0$.
\end{theorem}

The argument in \cite{Bourgain2004} is qualitative, but relies on induction on scales and orthogonality considerations similar to decoupling. Also, there is a key step which can be perceived as partial decoupling. Possibly the proof can be regarded as a qualitative precursor to efficient congruencing in two dimensions (cf. \cite{Wooley2016,Li2020,Li2021,GuoLiYungZorinKranich2021}). In the following we quantify the derivative gain in the trilinear estimate, which was left as an open question in \cite{Bourgain2004}:

\begin{theorem}
\label{thm:ImprovedStrichartzQuantified}
Let $0<\alpha \leq 1$. Under the assumptions \eqref{eq:SizeConstraint1} and \eqref{eq:FourierSupportConstraint} of Theorem \ref{thm:BourgainImprovedStrichartz}, we have
\begin{equation}
\label{eq:ImprovedStrichartzQuantified}
\int_{[0,N_1^{-\alpha}] \times \T } \prod_{j=1}^3 |e^{it \Delta} f_j |^2 dx dt \leq C \prod_{j=1}^3 \| f_j \|_2^2,
\end{equation}
where $C = C_{\varepsilon,\beta} \log(N_1)^{12+\varepsilon} N_1^{-\frac{\alpha \beta}{8}}$ for $\varepsilon > 0$.
\end{theorem}

For $\alpha = 1$, we can take advantage of short-time bilinear Strichartz estimates, for which we refer to works by Moyua--Vega \cite{MoyuaVega2008} in one dimension and Hani \cite{Hani2012} in higher dimensions. For $\alpha < 1$ we shall lower the effective size of the Fourier support by orthogonality considerations related to decoupling. This makes an application of Bernstein's inequality more favorable.

\medskip

The key new ingredient in the proof of Theorem \ref{thm:ImprovedStrichartzQuantified} is a decoupling iteration based on efficient congruencing, which Wooley employed for the first time in \cite{Wooley2016} to prove Vinogradov's Mean-Value Theorem in the cubic case. This was translated to prove decoupling for the paraboloid with an improved decoupling constant by Li \cite{Li2020,Li2021}.
  Since this decoupling iteration allows us to quantify the decoupling constant fairly precisely, we can show the further logarithmic refinement:
\begin{theorem}
\label{thm:LogarithmicRefinementTrilinearStrichartz1d}
Let $N_i \in 2^{\N}$ be dyadic numbers with $N_3 \leq N_1 \exp \big( \frac{- \log N_1}{\log \log N_1} \big)$ and $\alpha > 0$. Suppose that $\text{supp}(\hat{f}_i) \subseteq [N_i/2,N_i]$. Then the following estimate holds:
\begin{equation}
\label{eq:LogarithmicRefinement}
\int_{[0,N_1^{-\alpha}] \times \T } |e^{it \Delta} f_1|^2 |e^{it \Delta} f_2 |^2 | e^{it \Delta} f_3 |^2 dx dt \lesssim_\alpha \prod_{i=1}^3 \| f_i \|_2^2.
\end{equation}
\end{theorem}

\medskip

We remark that the currently best linear decoupling constant for the paraboloid is given by $\log(R)^c$ for $R \gg 1$:
\begin{equation*}
\| \mathcal{E} f \|_{L^6_{t,x}(\R^2)} \lesssim \log(R)^c \big( \sum_{\theta: R^{-\frac{1}{2}}-\text{ball}} \| \mathcal{E} f_\theta \|_{L^6_{t,x}(\R^2)}^2 \big)^{\frac{1}{2}}.
\end{equation*}
This was proved by Guth--Maldague--Wang \cite{GuthMaldagueWang2020} via the `High-Low method'; see also \cite{GuoLiYung2021} for an improvement in the context of discrete Fourier restriction. This approach might lead to further improvements of Theorem \ref{thm:ImprovedStrichartzQuantified}.

\medskip

Next we show a corresponding result in two dimensions using the original decoupling iteration due to Bourgain--Demeter \cite{BourgainDemeter2015} following the presentation by Demeter in \cite[Chapter~10]{Demeter2020}. The set-up is as follows:

Let $N_i \in 2^{\N}$, $i=1,2,3$, $N_3 \ll N_1$ with $N_3^{1+\beta} = N_1$ for some $0<\beta \leq 1$. Let $f_i: \T_\gamma^2 \to \C$ be periodic functions on $\T_\gamma^2 = \R / (2 \pi \Z) \times \R / (2 \pi \gamma \Z)$ with $\gamma \in (\frac{1}{2},1]$. We emphasize that the decoupling arguments do not distinguish between the square and irrational tori, which means we can allow for general $\T^2_\gamma$. 

\medskip

 We suppose that for $i=1,2,3$
\begin{equation} 
\label{eq:FourierSupport}
 \text{supp}(\hat{f}_i) \subseteq B(\xi_i^*,N_3) \text{ for some } \xi_i^* \in B(0,2N_1)
\end{equation}
and
\begin{equation}
\label{eq:SeparationCondition}
|\xi_1^* - \xi_2^*| \gtrsim N_1.
\end{equation}
Consider the expression
\begin{equation}
\label{eq:MultilinearExpression2d}
\int_{[0,N_1^{-\alpha}] \times \T_\gamma^2} |e^{it \Delta} f_1 \, e^{it \Delta} f_2 \, e^{it \Delta} f_3|^{4/3} dx dt.
\end{equation}
This can be regarded as a trilinear version of the Strichartz estimates in two dimensions on frequency-dependent times at the critical exponent $p=4$.

We impose the following transversality condition: Let $\xi_i \in \text{supp}(\hat{f}_i)$. Then we require that there is some $\nu \in (0,1]$ such that
 \begin{equation}
 \label{eq:Transversality3d}
 \xi_1, \; \xi_2, \; \xi_3 \text{ form a triangle of area greater than } N_1^2 \nu.
 \end{equation}
The reason for imposing this condition becomes evident after rescaling: We have for $\xi'_i \in \text{supp}(\hat{f}_i) /N_1$ that
\begin{equation*}
\xi'_1, \; \xi'_2, \; \xi_3' \text{ form a triangle of area greater than } \nu,
\end{equation*}
which is equivalent to the condition that the normals
\begin{equation*}
\mathfrak{n}_i = \frac{(-2 \xi'_i,1)}{|(-2 \xi'_i,1)|}
\end{equation*}
of the paraboloid $(\xi,-|\xi|^2)$ with $N \xi'_i \in \text{supp} (\hat{f}_i)$ are quantitatively transverse:
\begin{equation*}
|\mathfrak{n}_1 \wedge \mathfrak{n}_2 \wedge \mathfrak{n}_3| \gtrsim \nu.
\end{equation*}

\begin{theorem}
\label{thm:TrilinearStrichartzEstimates2d}
Let $\alpha \in (0,1]$. Under the above Fourier support and transversality assumptions on $f_i$, $i=1,2,3$ with $\| f_i \|_2 = 1$, the following estimate holds:
\begin{equation}
\label{eq:TrilinearEstimateA}
\int_{[0,N_1^{-\alpha}] \times \T_\gamma^2} \big| e^{it \Delta} f_1 e^{it \Delta} f_2 e^{it \Delta} f_3 \big|^{\frac{4}{3}} dx dt \lesssim_\varepsilon \nu^{-1+\frac{\beta}{2(1+\beta)}} N_1^{- \frac{\alpha \beta}{12} + \varepsilon }.
\end{equation}
\end{theorem}

\begin{remark}
By applying H\"older's inequality and the $L^4_{t,x}$-Strichartz estimate we obtain
\begin{equation*}
\begin{split}
\int_{[0,N^{-\alpha}] \times \T^2_\gamma} \prod_{i=1}^3 \big| e^{it \Delta} f_i \big|^{\frac{4}{3}} dx dt &\lesssim \big( \prod_{i=1}^3 \| e^{it \Delta} f_i \|_{L^4_{t,x}([0,N^{-\alpha}] \times \T^2_\gamma)} \big)^{\frac{4}{3}} \\
&\lesssim_\varepsilon N^{\varepsilon} \big( \prod_{i=1}^3 \| f_i \|_2 \big)^{\frac{4}{3}}.
\end{split}
\end{equation*}
Hence, if the transversality degenerates too much, then the multilinear estimates from Theorem \ref{thm:TrilinearStrichartzEstimates2d} cannot improve on the linear argument.
\end{remark}

\bigskip

We hope that the estimates proved in Theorem \ref{thm:ImprovedStrichartzQuantified} and Theorem \ref{thm:TrilinearStrichartzEstimates2d} can be used in the analysis of mass-critical nonlinear Schr\"odinger equations as pioneered in \cite{Bourgain2004}:
\begin{equation}
\label{eq:MassCriticalNLSIntroduction}
\left\{ \begin{array}{cl}
i \partial_t u + \Delta u &= \pm |u|^{\frac{4}{d}} u, \quad (t,x) \in \R \times \T^d, \\
u(0) &= u_0 \in H^s(\T^d).
\end{array} \right.
\end{equation}

In two dimensions, Herr--Kwak \cite{HerrKwak2023} recently reported how small mass global well-posedness for any $s>0$ follows from a refined $L^4_{t,x}$-Strichartz estimate. They show
\begin{equation}
\label{eq:2dLogarithmicStrichartz}
\| P_S e^{it \Delta} f \|_{L^4([0,(\log \# S)^{-1}],L^4(\T^2))} \lesssim \| f \|_{L^2(\T^2)}
\end{equation}
for any bounded set $S \subseteq \Z^2$ and that this implies global well-posedness for $s>0$. Above $P_S$ denotes the Fourier projection onto $S$. 

This motivates to study the short-time Strichartz estimate
\begin{equation*}
\| P_N e^{it \Delta} f \|_{L^p_{t,x}([0,N^{-\alpha}] \times \T^d)} \lesssim \| f \|_{L^2}
\end{equation*}
for $\alpha \in (0,1]$ at the critical exponent $p=\frac{2(d+2)}{d}$.
From this point of view the estimates from Theorems \ref{thm:ImprovedStrichartzQuantified} and \ref{thm:TrilinearStrichartzEstimates2d} give a quantified version of trilinear Strichartz estimates on frequency-dependent time intervals. In Section \ref{section:StrichartzEstimatesLargeTorus} we translate the multilinear refinements to rescaled tori $\lambda \T^d$ on the unit time interval. This is the form, in which multilinear Strichartz estimates were applied in the context of the $I$-method; see \cite{CollianderKeelStaffilaniTakaokaTao2003,FanStaffilaniWangWilson2018,
DeSilvaPavlovicStaffilaniTzirakis2007}. Notably, Fan \emph{et al.} \cite{FanStaffilaniWangWilson2018} proved bilinear Strichartz estimates on rescaled tori via decoupling, and the question about trilinear decoupling inequalities was raised. Although the multiscale inequalities derived in the forthcoming sections do not provide us with trilinear decoupling, the argument provides us with genuinely trilinear Strichartz estimates.

\bigskip

The question for global well-posedness of \eqref{eq:MassCriticalNLSIntroduction} motivates to study Strichartz estimates without derivative loss on frequency-dependent time intervals:
\begin{equation}
\label{eq:ShorttimeStrichartzSEQIntroduction}
\| P_N e^{it \Delta} f \|_{L_{t,x}^{p}([0,N^{-\alpha}] \times \T^d)} \lesssim \| f \|_{L^2(\T^d)}.
\end{equation}
at the critical exponent $p=\frac{2(d+2)}{d}$.

Recall that for $\alpha = 1$ \eqref{eq:ShorttimeStrichartzSEQIntroduction} is true on arbitrary compact Riemannian manifolds (without boundary) due to Burq--G\'erard--Tzvetkov \cite{BurqGerardTzvetkov2004} (see also Staffilani--Tataru \cite{StaffilaniTataru2002}). On the other hand, Bourgain \cite{Bourgain1993A} pointed out that \eqref{eq:ShorttimeStrichartzSEQIntroduction} fails for $\alpha = 0$ in one dimension; see also \cite{TakaokaTzvetkov2001} for higher dimensions. (Since $\ell^2$-decoupling implies discrete Fourier restriction, this yields that the constant $R^\varepsilon$ in \eqref{eq:l2DecouplingIntroduction} cannot be removed.) Kishimoto \cite{Kishimoto2014} furthermore pointed out that \eqref{eq:MassCriticalNLSIntroduction} cannot be analytically well-posed in $L^2(\T^d)$ for $d=1,2$.
We conjecture that, similarly to \eqref{eq:2dLogarithmicStrichartz}, the estimate
\begin{equation*}
\| P_N e^{it \Delta} f \|_{L^6_{t,x}([0,\log(N)^{-1}] \times \T)} \lesssim \| f \|_{L^2(\T)}
\end{equation*}
holds true. This would imply
\begin{equation*}
\| P_N e^{it \Delta} f \|_{L^6_{t,x}(\T \times \T)} \lesssim \log(N)^{\frac{1}{6}} \| f \|_{L^2(\T)},
\end{equation*}
which matches Bourgain's counterexample \cite{Bourgain1993A}.

\medskip

We analyze \eqref{eq:ShorttimeStrichartzSEQIntroduction} via decoupling and obtain smoothing estimates for $2 \leq p < \frac{2(d+2)}{d}$; see Proposition \ref{prop:ShorttimeStrichartzLowP}.
 For $p=\frac{2(d+2)}{d}$ we find smoothing as soon we increase the dispersion considering propagators $e^{it \Delta} \to e^{it (-\Delta)^{a/2}}$ for $a > 2$. Note that smoothing is not expected to hold for the Schr\"odinger propagator for $\alpha \in (0,2]$, not even in Euclidean space: For $\alpha = 0$ the Knapp example rules out smoothing and for $\alpha = 2$ we can choose highly localized initial data, which exhausts Sobolev embedding.
The smoothing estimates for the Airy propagator $e^{t \partial_x^3}$ on $\T$ is recorded in Proposition \ref{prop:ShorttimeStrichartzAiry}.

\medskip

Finally, in Section \ref{section:ExistenceMKDV} we use the Strichartz estimates for the Airy propagator on frequency-dependent time scales to prove the sharp existence result for the modified Korteweg-de Vries equation in Sobolev regularity. In \cite{Schippa2020} a conditional result on non-existence of the mKdV equation
\begin{equation}
\label{eq:MKDVIntroduction}
\left\{ \begin{array}{cl}
\partial_t u + \partial_x^3 u &= \pm u^2 \partial_x u, \quad (t,x) \in \R \times \T, \\
u(0) &= u_0 \in H^s(\T)
\end{array} \right.
\end{equation}
was proved, which states that existence of solutions to \eqref{eq:MKDVIntroduction} only holds for the Wick ordered nonlinearity for $s<0$. With the smoothing estimates from Section \ref{section:LinearSmoothing} at hand, the non- existence result for negative Sobolev regularity follows. This points out that the established short-time Strichartz estimates are also relevant in applications to quasilinear equations. We refer to \cite{SchippaPhD2019} for further reading.

\medskip

\emph{Outline of the paper.} In Section \ref{section:Preliminaries} we prepare the proof of Theorem \ref{thm:ImprovedStrichartzQuantified} by introducing notations and giving a detailed overview of the argument. This is made rigid in Section \ref{section:ProofTheoremB}. In Section \ref{section:LogarithmicRefinement} we show the logarithmic improvement in Theorem \ref{thm:LogarithmicRefinementTrilinearStrichartz1d}. In Section \ref{section:TrilinearStrichartz2d} we prove Theorem \ref{thm:TrilinearStrichartzEstimates2d}. In Section \ref{section:StrichartzEstimatesLargeTorus} we translate the refined Strichartz estimates to estimates on tori with large period at finite times. Next, we show new linear short-time Strichartz estimates with smoothing effect in Section \ref{section:LinearSmoothing}, which are applied to prove the sharp existence result for the modified Korteweg-de Vries equation in Section \ref{section:ExistenceMKDV}. In the Appendix we show how the Fefferman--Córdoba square function estimate implies the bilinear short-time Strichartz estimate in one dimension.

\medskip

\textbf{Basic notations:}

\begin{itemize}
\item $\T = \R / (2 \pi \Z)$ denotes the torus in one dimension,
\item $(x,t)$ denotes space-time variables and $\xi$ denotes the dual frequency variable to $x$,
\item $\hat{f}$ denotes the Fourier transform of $f: \T^d \to \C$:
\begin{equation*}
\hat{f}(k) =  \int_{\T^d} e^{-i x k} f(x) dx, \quad k \in \Z^d.
\end{equation*}
For a function $F: \R \times \T^d \to \C$ we denote with $\hat{F}$ the space-time Fourier transform.
\item $H^s(\T^d)$ and $H^s(\R^d)$ denote the $L^2$-based Sobolev spaces with Sobolev regularity $s \in \R$. The Sobolev norms are given by
\begin{equation*}
\| f \|_{H^s(\T^d)}^2 = \sum_{k \in \Z^d} \langle k \rangle^{2s} |\hat{f}(k)|^2, \quad \| f \|_{H^s(\R^d)}^2 = \int_{\R^d} \langle \xi \rangle^{2s} |\hat{f}(\xi)|^2 d\xi.
\end{equation*}
\item We denote with $P_N$ smooth Littlewood-Paley projections to frequencies of modulus in $[N/2,2N]$ for $N \in 2^{\N}$. $P_1$ denotes smooth frequency projection to the ball of radius $2$. Capital numbers usually denote dyadic numbers $N \in 2^{\N_0}$.
\item $A \lesssim B$ means there is a harmless constant $C >0$ such that $A \leq C B$. Further dependencies of the constant $C$ are indicated with subindices, e.g. $A \lesssim_\varepsilon B$ means there is $C_\varepsilon = C(\varepsilon)$ such that $A \leq C_\varepsilon B$.
\end{itemize}

\section{Preliminaries}
\label{section:Preliminaries}

In this section we prepare the proof of Theorem \ref{thm:ImprovedStrichartzQuantified}. We collect basic arguments for future reference, which are presently described in detail.

We shall use the currently best result on linear discrete Fourier restriction for the paraboloid due to Guo--Li--Yung:
\begin{theorem}[{\cite[Theorem~1.1]{GuoLiYung2021}}]
\label{thm:DiscreteStrichartzL6}
For every $\varepsilon > 0$, there exists a constant $C_\varepsilon > 0$ such that
\begin{equation*}
\| \sum_{k=1}^N f_n e^{i (kx - k^2 t)} \|_{L^6(\T^2)} \leq C_\varepsilon \big( \log(N) \big)^{2+ \varepsilon} \big( \sum_{k=1}^N |f_k|^2 \big)^{\frac{1}{2}}.
\end{equation*}
\end{theorem}
It is conjectured (based on an example due to Bourgain \cite{Bourgain1993A}) that the exponent can be lowered to $1/6$. Note that this would not give an essential improvement of the present result.

\subsection{Initial reductions}
\label{subsection:InitialReductions}
We can normalize $\| f_j \|_{L^2} = 1$. Next, we apply H\"older's inequality to find from Theorem \ref{thm:DiscreteStrichartzL6}
\begin{equation}
\label{eq:InitialReduction1d}
\begin{split}
\int_{[0,N_1^{-\alpha}] \times \T} |e^{it \Delta} f_1 e^{it \Delta} f_2 e^{it \Delta} f_3|^2 dx dt &\lesssim \big( \int_{[0,N_1^{-\alpha}] \times \T} |e^{it \Delta} f_1|^2 |e^{it \Delta} f_3|^4 \big)^{\frac{1}{2}} \\
&\quad \times \big( \int_{[0,N_1^{-\alpha}] \times \T} |e^{it \Delta} f_2|^6 \big)^{\frac{1}{3}} \big( \int_{[0,N_1^{-\alpha}] \times \T} |e^{it \Delta} f_1|^6 \big)^{\frac{1}{6}} \\
&\lesssim_\varepsilon \log(N_1)^{6+\varepsilon} \big( \int_{[0,N_1^{-\alpha}] \times \T} |e^{it \Delta} f_1|^2 |e^{it \Delta} f_3|^4 \big)^{\frac{1}{2}}.
\end{split}
\end{equation}

By almost orthogonality we decompose $f_1 = \sum_{i_1} f_{1,i_1}$ into functions with Fourier support contained in intervals of size $N_3$ and assume in the following that $f_1$ is Fourier supported in an interval of size $N_3$. It suffices to prove an estimate
\begin{equation}
\label{eq:OrthogonalReduction}
\int_{[0,N_1^{-\alpha}] \times \T} |e^{it \Delta} f_{1}|^2 |e^{it \Delta} f_{3}|^4  \lesssim \, C' \, \|f_1\|_2^2 \| f_3 \|_2^4
\end{equation}
with $C' = C_{\varepsilon,\beta} \log(N_1)^{12+\varepsilon} N_1^{-\frac{\alpha \beta}{8}}$ taking into account the initial reduction in \eqref{eq:InitialReduction1d}. To ease notation, redenote $f_3$ with $f_2$.

\medskip

The left hand-side in \eqref{eq:OrthogonalReduction} will be the outset of the asymmetric decoupling due to Li \cite{Li2020} after continuous approximation (see also follow-up works \cite{Li2021,GuoLiYungZorinKranich2021}).

\subsection{Continuous approximation}
\label{subsection:ContinuousApproximation1d}
We denote $N=N_1$ for brevity and
\begin{equation*}
\hat{f}_1(k)= a_{1,k}, \quad \hat{f}_2(k)= a_{2,k}.
\end{equation*}
In the first step we use periodicity and rescaling such that we integrate over a space-time ball of size $N^{2-\alpha}$. We have
\begin{equation}
\label{eq:BilinearExpression}
\begin{split}
&\quad \| e^{it \Delta} f_1 |e^{it \Delta} f_2|^2 \|_{L_{t,x}^2([0,N^{-\alpha}] \times \T)}^2 \\
&= \big\| \sum_{k_1 \in \Z} e^{i(k_1 x - t k_1^2)} a_{1,k_1} \big( \sum_{k_2 \in \Z} e^{i(k_2 x - t k_2^2)} a_{2,k_2} \big)^2 \big\|_{L_{t,x}^2([0,N^{-\alpha}] \times \T)}^2.
\end{split}
\end{equation}
We carry out the changes of variables $x \to \frac{x}{N}$, $k_2 \to N k_2$, $t \to N^{-2} t$, and denote $b_{i,k_i} = a_{i,N k_i}$ accordingly:
\begin{equation*}
\begin{split}
\eqref{eq:BilinearExpression} &= N^{-1} \big\| \big( \sum_{\substack{k_1 \in \Z / N, \\ k_1 \in I_1}} e^{i(k_1 x - N^2 t k_1^2)} b_{1,k_1} \big) \big( \sum_{\substack{k_2 \in \Z / N, \\ k_2 \in I_2}} e^{i(k_2 x - N^2 t k_2^2)} b_{2,k_2} \big)^2 \|_{L_{t,x}^2([0,N^{-\alpha}] \times N \T)}^2 \\
&= N^{-3} \big\| \big( \sum_{\substack{k_1 \in \Z / N, \\ k_1 \in I_1}} e^{i(k_1 x - N^2 t k_1^2)} b_{1,k_1} \big) \big( \sum_{\substack{k_2 \in \Z / N, \\ k_2 \in I_2}} e^{i(k_2 x - N^2 t k_2^2)} b_{2,k_2} \big)^2 \|_{L_{t,x}^2([0,N^{2-\alpha}] \times N \T)}^2.
\end{split}
\end{equation*}

Next, we use periodicity in $x$:
\begin{equation*}
\eqref{eq:BilinearExpression} = N^{-(4-\alpha)} \big\| \big| \sum_{\substack{k_1 \in \Z/N, \\ k_1 \in I_1}} e^{i(k_1 x - t k_1^2)} b_{1,k_1} \big| \big| \sum_{\substack{k_2 \in \Z/N,\\ k_2 \in I_2}} e^{i(k_2 x - t k_2^2)} b_{2,k_2} \big|^2 \big\|^2_{L_{t,x}^2([0,N^{2-\alpha}]^2 )}.
\end{equation*}
The intervals $I_i$ are of length less than $N^{-\frac{\beta}{1+\beta}}$ and have a distance $\gtrsim 1$. We obtain a factor $N^{-4+\alpha}$ from the continuous approximation. When we reverse the continuous approximation, the inverse of this factor will come up. In the following we focus on
\begin{equation*}
\big\| \big| \sum_{\substack{k_1 \in \Z/N, \\ k_1 \in I_1}} e^{i(k_1 x - t k_1^2)} b_{1,k_1} \big| \big| \sum_{\substack{k_2 \in \Z/N,\\ k_2 \in I_2}} e^{i(k_2 x - t k_2^2)} b_{2,k_2} \big|^2 \big\|^2_{L_{t,x}^2([0,N^{2-\alpha}]^2 )}.
\end{equation*}

Now we approximate the exponential sums by Fourier extension operators:
\begin{equation*}
| \sum e^{i (k_1 x - t k_1^2)} b_{1,k_1} | \big| \sum e^{i(k_2 x - t k_2^2)} b_{2,k_2} \big|^2  \lesssim | \mathcal{E} \tilde{f}_1 | |\mathcal{E} \tilde{f}_2|^2  \text{ on } B_{N^{2-\alpha}}
\end{equation*}
with
\begin{equation*}
\mathcal{E} \tilde{f}(x,t) = \int_{\{ |\xi| \leq 1 \}} e^{i(x\cdot \xi - t |\xi|^2)} \tilde{f}(\xi) d\xi.
\end{equation*}
The support in $\xi$ for $\tilde{f}_i$ is given by $I_i$. With $\psi \in C^\infty_c(B(0,1))$, $\psi \geq 0$, $\int \psi dx = 1$, we have
\begin{equation*}
\tilde{f}_i(\xi) = \sum_{k \in \Z/N} b_{i,k} \lambda \psi(\lambda(\xi-k))
\end{equation*}
such that for $\lambda \to \infty$ we obtain in $\mathcal{S}'(\R)$
\begin{equation*}
\tilde{f}_i(\xi) = \sum_{k \in \Z/N} b_{i,k} \delta(\xi-k).
\end{equation*}
In the following suppose that $\lambda \geq N$ is large enough  such that the above expression becomes
\begin{equation*}
\begin{split}
&\quad \big\| \big| \sum_{\substack{k_1 \in \Z/N, \\ k_1 \in I_1}} e^{i(k_1 x - t k_1^2)} b_{1,k_1} \big| \big| \sum_{\substack{k_2 \in \Z/N,\\ k_2 \in I_2}} e^{i(k_2 x - t k_2^2)} b_{2,k_2} \big|^2 \big\|^2_{L_{t,x}^2([0,N^{2-\alpha}] \times [0,N^{2-\alpha}])} \\
&\lesssim \| \mathcal{E} \tilde{f}_1 |\mathcal{E} \tilde{f}_2|^2 w^3_{B_{N^{2-\alpha}}} \|_{L^2(\R^2)}^2.
\end{split}
\end{equation*}

Above we let $w_{B_{N^{2-\alpha}}}$ be a Schwartz weight for $B_{N^{2-\alpha}}$, which is rapidly decaying off the ball and compactly Fourier supported. Note that multiplication with $w_{B_{N^{2-\alpha}}}$ is equivalent to blurring the Fourier support on the scale $N^{-2+\alpha}$. $\mathcal{E} \tilde{f}_i \cdot w_{B_{N^{2-\alpha}}} = g_i$ has (space-time) Fourier support in the $N^{-(2-\alpha)}$-neighbourhood of $\{(\xi,-\xi^2) : \xi \in I_i \}$.

\medskip

By reversing the continuous approximation we mean taking the limit $\lambda \to \infty$ such that the Fourier extension operator becomes an exponential sum.  Note that the continuous approximation must be reversed on the level of extension operators or equivalently, space-time estimates. On the level of $L_x^2$-norms it is too costly.

\medskip

We would like to decouple to the scale $N^{-1}$ (since we rescaled by $N$, the spacing between the frequencies is now $1/N$). However, the finest spatial scale we can reach by decoupling respecting the weight function is given by $N^{-1+\alpha/2}$. At this scale we exit the decoupling iteration.

Let $\pi_j: \R^2 \to \R$, $x \mapsto x_j$ denote coordinate projections and $I_j \subseteq [0,1]$ intervals. In the following we aim to decouple the expression
\begin{equation*}
\int |g_{I_1}|^4 |g_{I_2}|^2
\end{equation*}
with $\text{supp} (\pi_1(\hat{g}_{I_j})) \subseteq I_j \subseteq [0,1]$ with $|I_j| \lesssim N^{-\frac{\beta}{1+\beta}}$ and $\text{dist}(I_1,I_2) \sim 1$. In the following we indicate the spatial frequency support by intervals as subindices and the weight $w_{B_{N^{2-\alpha}}}$ is absorbed into the functions.

\subsection{Overview of the decoupling argument}
\label{subsection:Overview}
In the following we carry out a decoupling iteration inspired by efficient congruencing (cf. \cite{Wooley2016,Li2020,Li2021,GuoLiYungZorinKranich2021}). Let $I_1$ and $I_2$ be intervals of length $\tilde{\delta}$. By an almost orthogonality argument based on transversality of the separated regions of the paraboloid, we have morally
\begin{equation*}
\int |g_{I_1}|^4 |g_{I_2}|^2 \lesssim \sum_{\substack{I_2^2 \subseteq I_2, \\ |I_2^2| = \tilde{\delta}^2}} \int |g_{I_1}|^4 |g_{I^2_2}|^2
\end{equation*}
with intervals $I_2^2$ partitioning $I_2$. Now we exchange the roles of $g_{I_1}$ and $g_{I_2}$ by using Cauchy-Schwarz:
\begin{equation}
\label{eq:FirstDecoupling}
\sum_{\substack{I_2^2 \subseteq I_2, \\ |I_2^2| = \tilde{\delta}^2}} \int |g_{I_1}|^4 |g_{I^2_2}|^2 \lesssim \| g_{I_1} \|_{L^6}^3 \sum_{\substack{I_2^2 \subseteq I_2, \\ |I_2^2| = \tilde{\delta}^2}} \big( \int |g_{I_1}|^2 |g_{I^2_2}|^4 \big)^{\frac{1}{2}}.
\end{equation}
The second expression can be decoupled further. The first expression loses a logarithmic factor, which is acceptable.
For the first expression we reverse the scaling:
\begin{equation*}
\| \mathcal{E} \tilde{f}_{I_1} w_{B_{N^{2-\alpha}}} \|^6_{L^6_{t,x}(\R^2)} \lesssim N^3 \| \mathcal{E} \tilde{f}_{I_1} w_{N^{-\alpha}}(t) w_{N^{1-\alpha}}(x) \|_{L^6_{t,x}}^6.
\end{equation*}
Now we take the limit $\lambda \to \infty$ to recover the exponential sum (``reversing the continuous approximation"):
\begin{equation*}
\big\| \sum_{k \in \Z} a_{i,k} e^{i(kx - tk^2)} w_{N^{-\alpha}}(t) w_{N^{1-\alpha}}(x) \|^6_{L^6_{t,x}(\R^2)}.
\end{equation*}
Since we are proving $L^2$-based estimates, it suffices to show an estimate for
\begin{equation*}
N^3 \big\| \sum_{k \in \Z} a_{i,k} e^{i(kx - tk^2)} \big\|^6_{L^6_{t}([0,N^{-\alpha}],L^6_x([0,N^{1-\alpha}]))}.
\end{equation*}
By periodicity, we obtain
\begin{equation*}
\begin{split}
&\quad N^3 \big\| \sum_{k \in \Z} a_{i,k} e^{i(kx - tk^2)} \big\|^6_{L^6_{t}([0,N^{-\alpha}],L^6_x([0,N^{1-\alpha}]))} \\
 &= N^{4-\alpha} \big\| \sum_{k \in \Z} a_{i,k} e^{i(kx - tk^2)} \big\|^6_{L^6_{t}([0,N^{-\alpha}],L^6_x(\T))} \lesssim N^{4-\alpha} \log(N)^{12+6 \varepsilon} \| a_i \|_{\ell^2}^6.
\end{split}
\end{equation*}
Roughly speaking, whenever we reverse the continous approximation and switch back to the periodic setting, we get a factor $N^{4-\alpha}$. This is compensated by the factor $N^{-4+\alpha}$ from the initial change of variables to normalize to unit frequencies.

\medskip

With any iteration we can lower the scale from $\tilde{\delta}$ to $\tilde{\delta}^2$. So, in order to reach the interval of length $N^{-1}$ (actually we can only reach $N^{-1+\frac{\alpha}{2}}$ due to time localization), $M$ iterations suffice with $M$ given by
\begin{equation*}
M = \min\{ \ell \in \N : N^{-\frac{\beta}{1+\beta} 2^{\ell}} = N^{-1} \}.
\end{equation*}
We choose as upper bound for the number of iterations
\begin{equation*}
m = \lceil \log_2(\beta^{-1}+1) \rceil.
\end{equation*}

We can estimate $\| g_{I_1} \|_{L^6}$ by reversing the approximation steps for periodic functions, which incurs by Strichartz estimates $\lesssim_\varepsilon \log(N_1)^{2+\varepsilon}$. 

 We come back to \eqref{eq:FirstDecoupling}. It suffices to prove an estimate
\begin{equation*}
\int |g_{I_1}|^2 |g_{I_2^2}|^4 \lesssim C N^{4-\alpha} \| b_{I_1} \|_{\ell^2}^2 \| b_{I_2^2} \|_{\ell^2}^4.
\end{equation*}
because we obtain by almost orthogonality
\begin{equation*}
\begin{split}
\sum_{I^2_2} \big( \int |g_{I_1}|^2  |g_{I_2}|^4 \big)^{\frac{1}{2}} &\leq C^{\frac{1}{2}} N^{4-\alpha} \| b_{I_1} \|_{\ell^2} \sum_{I^2_2} \| b_{I_2^2} \|_{\ell^2}^2 \\
&\leq C^{\frac{1}{2}} N^{4-\alpha} \| b_{I_1} \|_{\ell^2} \| b_{I_2} \|^2_{\ell^2}.
\end{split}
\end{equation*}
It is important to sum the $\ell^2$-norms of the coefficients (after reversing the continuous approximation on the level of space-time estimates). We use this decoupling iteration until we reach an expression of the form
\begin{equation*}
\int |g_{\tilde{I}_1}|^2 |g_{\tilde{I}_2}|^4
\end{equation*}
with $|\tilde{I}_1| = |\tilde{I}_2| = N^{-1+\frac{\alpha}{2}}$. We again reverse the continuous approximation to find
\begin{equation*}
\int |g_{\tilde{I}_1}|^2 |g_{\tilde{I}_2}|^4 \lesssim N^{4-\alpha} \int_{[0,N^{-\alpha}] \times \T}  \big| \sum_{k \in \Z} a_{1,k} e^{i(kx - tk^2)} \big|^2 \big| \sum_{k \in \Z} a_{2,k} e^{i(kx-tk^2)} \big|^4 dx dt.
\end{equation*}
We use a short-time bilinear Strichartz estimate due to Moyua--Vega \cite{MoyuaVega2008}. Hani \cite{Hani2012} proved this on general compact manifolds. In the Appendix we show how a transverse (reverse) square function estimate going back to Fefferman and Córdoba \cite{Fefferman1973,Cordoba1982} implies the following by the arguments of the paper:
\begin{theorem}[{\cite[Theorem~1.2]{Hani2012}}]
Suppose that $u_0$, $v_0 \in L^2(\T)$ with $P_{N_1} u_0 = u_0$ and $P_{N_2} v_0= v_0$ with $N_2 \ll N_1$. Then the following estimate holds:
\begin{equation}
\label{eq:ShorttimeBilinearStrichartz}
\| e^{it \Delta} u_0 \, e^{it \Delta} v_0 \|_{L^2([0,N_1^{-1}],L^2(\T))} \lesssim N_1^{-\frac{1}{2}} \| u_0 \|_{L^2(\T)} \| v_0 \|_{L^2(\T)}.
\end{equation}
\end{theorem}

By \eqref{eq:ShorttimeBilinearStrichartz} and Bernstein's inequality to conclude a gain:
\begin{equation*}
\begin{split}
&\quad \int_{[0,N^{-\alpha}]} \int_{\T} \big| \sum_{k \in \Z} a_{1,k} e^{i(kx - tk^2)} \big|^2 \big| \sum_{k \in \Z} a_{2,k} e^{i(kx-tk^2)} \big|^4 dx dt \\
 &\lesssim \| e^{it \Delta} f_1 e^{it \Delta} f_2 \|^2_{L^2([0,N^{-\alpha}],L^2(\T))} \| e^{it \Delta} f_2 \|_{L^\infty}^2 \\
&\lesssim N_1^{-\alpha} \| f_1 \|_2^2 \| f_2 \|_2^2 N_1^{\frac{\alpha}{2}} \| f_2 \|_2^2 \\
&\lesssim N_1^{-\frac{\alpha}{2}} \| f_1 \|_2^2 \| f_2 \|_2^4.
\end{split}
\end{equation*}
Unwinding the decoupling iteration, this will yield Theorem \ref{thm:ImprovedStrichartzQuantified}. In the next section we turn to the technical details.

\section{Proof of Theorem \ref{thm:ImprovedStrichartzQuantified}}
\label{section:ProofTheoremB}
Following the reductions in Section \ref{section:Preliminaries}, let $I_j \subseteq [0,1]$, $j=1,2$ be two intervals, which satisfy $\text{dist}(I_1,I_2) \gtrsim 1$, and $\tilde{\delta} = |I_i| \leq N^{-\frac{\beta}{1+\beta}}$. Let $g_i \in \mathcal{S}(\R^2)$ with $\text{supp}(\hat{g}_i) \subseteq \mathcal{N}_{N^{-2+\alpha}}(\{(\xi,-\xi^2) : \xi \in I_i \})$.
We analyze the following expression via \textit{asymmetric bilinear decoupling} (cf. \cite{GuoLiYungZorinKranich2021}):
\begin{equation}
\label{eq:AsymmetricBilinearDecoupling}
\int_{\R^2} |g_{I_1}|^2 |g_{I_2}|^4 dx.
\end{equation}
More specifically, we have $g_I = \mathcal{E} b_I \cdot w_{B_{N^{2-\alpha}}}$. This means that $g_I$ is given by the continuous approximation on the interval $I$ with Fourier support blurred on the scale $N^{-(2-\alpha)}$. We recall notations from \cite[Section~2]{GuoLiYungZorinKranich2021}.

Let $\Gamma(t) = (t,-t^2)$. For an interval $I \subseteq [0,1]$ with $|I| = \tilde{\delta}$, we let 
\begin{equation*}
\mathcal{U}_I = \{ (c_I,c_I^2) + a \cdot \Gamma'(c_I) + b \cdot \Gamma''(c_I) \, : \, |a| \leq \tilde{\delta}, \; |b| \leq \tilde{\delta}^2 \}.
\end{equation*}
This describes roughly the $\tilde{\delta}^2$-neighbourhood of the paraboloid in the interval $I$.
By $\mathcal{U}_I^o$ we denote the polar set of $\mathcal{U}_I$ centered at the origin:
\begin{equation*}
\mathcal{U}_I^o = \{ x \in \R^2 : | \langle x, \partial^i \Gamma(c_I) \rangle| \leq |I|^{-i}, \; 1 \leq i \leq 2 \}.
\end{equation*}
We let
\begin{equation*}
\phi_I(x) = |\mathcal{U}_I^o|^{-1} \inf \{ t \geq 1 \, : \, x/t \in \mathcal{U}_I^o \}^{-20}
\end{equation*}
such that $\phi_I \geq 0$ and $ \| \phi_I \|_{L^1} \lesssim 1$. We have the following consequence of the uncertainty relation:
\begin{lemma}[{\cite[Lemma~3.3]{GuoLiYungZorinKranich2021}}]
\label{lem:Uncertainty}
For $p \in [1,\infty)$ and $J \subseteq [0,1]$, we have
\begin{equation*}
|g_J|^p \lesssim_{p,C} |g_J|^p * \phi_J
\end{equation*}
for any $g_J$ with $\text{supp}(\hat{g}_J) \subseteq C \mathcal{U}_J$.
\end{lemma}

\medskip

Coming back to \eqref{eq:AsymmetricBilinearDecoupling}, we have
\begin{equation*}
\int |g_{I_1}|^4 |g_{I_2}|^2 \lesssim \int (|g_{I_1}|^4 * \phi_{I_1}) (|g_{I_2}|^2 * \phi_{I_2}).
\end{equation*}
At this point we use the lower-dimensional decoupling (based on Plancherel and transversality of pieces of the paraboloid on the two separated intervals $I_1$, $I_2$). The following estimate is \cite[Lemma~3.8]{GuoLiYungZorinKranich2021} with $l =1$ and $k=2$:
\begin{equation*}
\int (|g_{I_1}|^4 * \phi_{I_1}) ( |g_{I_2}|^2 * \phi_{I_2} ) \lesssim \sum_{\substack{I_2^2 \subseteq I_2, \\ |I_2^2| = \tilde{\delta}^2}} \int \big( |g_{I_1}|^4 * \phi_{I_1} \big) \big( |g_{I_2^2}|^2 * \phi_{I_2^2} \big). 
\end{equation*}

Applying H\"older's inequality (see \cite[Lemma~4.1]{GuoLiYungZorinKranich2021}) yields
\begin{equation*}
\sum_{|I_2^2| = \tilde{\delta}^2} \int \big( |g_{I_1}|^4 * \phi_{I_1} \big) \big( |g_{I_2^2}|^2 * \phi_{I_2^2} \big)\lesssim \| g_{I_1} \|_{L^6}^3 \sum_{|I_2^2| = \tilde{\delta}^2} \big( \int ( |g_{I_1}|^2 * \phi_{I_1} )( |g_{I_2^2}|^4 * \phi_{I_2^2} ) \big)^{\frac{1}{2}}.
\end{equation*}
We summarize our findings as
\begin{equation*}
\int |g_{I_1}|^4 |g_{I_2}|^2 \leq C_1 \| g_{I_1} \|_6^3 \sum_{|I_2^2| = \tilde{\delta}^2} \big( \int \big( |g_{I_1}|^2 * \phi_{I_1} \big) \big( |g_{I_2^2}|^4 * \phi_{I_2^2} \big) \big)^{\frac{1}{2}},
\end{equation*}
which is amenable to $m$-fold iteration:
\begin{equation}
\label{eq:DecouplingIteration}
\begin{split}
&\quad \int |g_{I_1}|^4 |g_{I_2}|^2 \\
 &\leq C_1 \| g_{I_1} \|_6^3 \sum_{|I_2^2| = \tilde{\delta}^2} \big( \int \big( |g_{I_1}|^2 * \phi_{I_1} \big) \big( |g_{I^2_2}|^4 * \phi_{I_2} \big) \big)^{\frac{1}{2}} \\
&\leq C_1^{\frac{3}{2}} \| g_{I_1} \|^3_6 \sum_{|I_2^2| = \tilde{\delta}^2} \big( \sum_{|I_1^2| = \tilde{\delta}^4} \int \big( |g_{I_1^2}|^2 * \phi_{1^2} \big) \big( |g_{I_2^2}|^4 * \phi_{I_2^2} \big) \big)^{\frac{1}{2}} \\
&\vdots \\
&\leq C_1^2 \| g_{I_1} \|^3_{L^6} \sum_{|I_2^2| = \tilde{\delta}^2} \| g_{I_2^2} \|_{L^6}^{\frac{3}{2}} \\
&\quad \times \big( \sum_{|I_1^2| = \tilde{\delta}^4} \| g_{I_1^2} \|_{L^6}^{\frac{3}{2}} \big( \sum_{|I_2^3| = \tilde{\delta}^4} \| g_{I_2^3} \|_{L^6}^{3/2} \ldots \big( \int |g_{I_1^{m/2}}|^4 * \phi_{I_1^{m/2}} |g_{I_2^{m/2}}|^2 * \phi_{I_2^{m/2}} \big)^{\frac{1}{2}} \big)^{\frac{1}{2}} \ldots \big)^{\frac{1}{2}}.
\end{split}
\end{equation}
Recall that $m = \lceil \log_2(\beta^{-1} + 1) \rceil$ (which we suppose to be even for notational convenience) for $\tilde{\delta} = N^{-\frac{\beta}{1+\beta}}$ such that $|I_j^{m/2}| = N^{-1+\frac{\alpha}{2}}$. By the argument in Section \ref{subsection:Overview} we have the following:
\begin{lemma}
\label{lem:LinearExitLemma}
Let $g_{I_j^k} = \mathcal{E} b_{I_j^k} \cdot w_{B_{N^{2-\alpha}}}$ be like above. Then
\begin{equation*}
\| g_{I_j^k} \|_{L^6(\R^2)} \leq N_1^{\frac{4-\alpha}{6}} \cdot C_\varepsilon \log(N_1)^{2+\varepsilon} \| b_{I_j^k} \|_{\ell^2}.
\end{equation*}
\end{lemma}
To estimate the maximally decoupled expression, we use the following:
\begin{lemma}
\label{lem:BilinearExitLemma}
Let $\tilde{I}_i \subseteq [0,1]$ with $|\tilde{I}_i| = N^{-1+\frac{\alpha}{2}}$ and assume the separation property
\begin{equation*}
\text{dist}(\tilde{I}_1,\tilde{I}_2) \sim 1.
\end{equation*}
Then the following estimate holds:
\begin{equation}
\label{eq:AsymmetricExitEstimate}
\int (|g_{\tilde{I}_1}|^2 * \phi_{\tilde{I}_1}) ( |g_{\tilde{I}_2}|^4 * \phi_{\tilde{I}_2} ) \leq N^{4-\alpha} \cdot D N^{-\frac{\alpha}{2}} \| b_{\tilde{I}_1} \|_{\ell^2}^2 \| b_{\tilde{I}_2} \|_{\ell^2}^4.
\end{equation}
\end{lemma}
\begin{proof}
By H\"older's inequality it suffices to estimate
\begin{equation*}
\int \big( |g_{\tilde{I}_1}|^2 * \phi_{\tilde{I}_1} |g_{\tilde{I}_2}|^2 * \phi_{\tilde{I}_2} \big) \| g_{I_2} \|_{L^\infty}^2.
\end{equation*}
We rewrite the first factor as
\begin{equation*}
\iint \big( \int |g_{\tilde{I}_1}|^2 (x-x_1,t-t_1) |g_{\tilde{I}_2}|^2(x-x_2,t-t_2) dx dt \big) \phi_1(x_1,t_1) \phi_2(x_2,t_2) dx_1 dx_2 dt_1 dt_2.
\end{equation*}
We shall prove a uniform bound in $x_1$,$x_2$ and $t_1$,$t_2$ for the integral over $x$ and $t$. To this end, we absorb the pure phase factor $e^{i(x_1-x_2) \xi} e^{i(t_1-t_2) \xi^2}$ into $g_2$ and dominate
\begin{equation*}
w_{B_{N^{2-\alpha}}}(x',t') w_{B_{N^{2-\alpha}}}(x'+x_1-x_2,t'+t_1-t_2) \lesssim w_{B_{N^{2-\alpha}}}(x',t').
\end{equation*}
We reverse the continuous approximation to find
\begin{equation*}
\begin{split}
&\quad \int \big| \int e^{i(x' \xi - t' \xi^2)} f_1(\xi) d\xi \int e^{i(x' \xi - t \xi^2)} \tilde{f}_2(\xi) d\xi \big|^2 w_{B_{N^{2-\delta}}}(x',t') dx'dt' \\
 &\lesssim N^{4-\alpha} \| e^{it \Delta} f_{\tilde{I}_1} e^{it \Delta} f'_{\tilde{I}_2} \|^2_{L_{t,x}^2([0,N^{-\alpha}] \times \T)}.
 \end{split}
\end{equation*}
By $f'$ we indicate that we have absorbed the phase shift into $f$ and write $\tilde{I}_j = N I_j$.
Note that $\| f'_{\tilde{I}_2} \|_2 = \| f_{\tilde{I}_2} \|_2$ because the phase shift leaves the $\ell^2$-norm invariant.

We piece together \eqref{eq:ShorttimeBilinearStrichartz} on intervals of size $N^{-1}$ to obtain
\begin{equation}
\label{eq:ShorttimeExit}
\| e^{it \Delta} f_{\tilde{I}_1} e^{it \Delta} f_{\tilde{I}_2} \|^2_{L_{t,x}^2([0,N^{-\alpha}] \times \T)} \leq D_1 N^{-\alpha} \| a_{I_1} \|_{\ell^2}^2 \| a_{\tilde{I}_2} \|_{\ell^2}^2.
\end{equation}
We estimate $\| f_{\tilde{I}_2} \|_{L^\infty}$ by Bernstein's inequality:
\begin{equation}
\label{eq:BernsteinExit}
\| f_{\tilde{I}_2} \|_{L^\infty}^2 \leq D_2 N^{\frac{\alpha}{2}} \| a_{\tilde{I}_2} \|_{\ell^2}^2.
\end{equation}
Taking \eqref{eq:ShorttimeExit} and \eqref{eq:BernsteinExit} together with $\| \phi_I \|_{L^1} \lesssim 1$, we conclude the proof of \eqref{eq:AsymmetricBilinearDecoupling}.
\end{proof}

We are ready to estimate the expression in \eqref{eq:DecouplingIteration}: The expressions $\| g_{I_j^k} \|_{L^6}$ are estimated by Lemma \ref{lem:LinearExitLemma}, which yields a factor of $C_\varepsilon \log(N_1)^{2+\varepsilon}$ at every instance. Note that the exponents in \eqref{eq:DecouplingIteration} sum up to a number $\leq 6$ by finite geometric sum. The bilinear expression is estimated by Lemma \ref{lem:BilinearExitLemma}, which yields
\begin{equation*}
\begin{split}
\eqref{eq:DecouplingIteration} &\leq C_1^2 (C_\varepsilon \log(N_1)^{2+\varepsilon})^6 (D N_1^{-\frac{\alpha}{2}})^{\frac{1}{2^m}} \| b_{I_1} \|^3_{\ell^2} \sum_{|I_2^2| = \tilde{\delta}^2} \| b_{I_2^2} \|_{\ell^2}^{\frac{3}{2}} \\
&\quad \times \big( \sum_{|I_1^2| = \tilde{\delta}^4} \| b_{I_1^2} \|_{\ell^2}^{\frac{3}{2}} \big( \sum_{|I_2^3| = \tilde{\delta}^4} \| b_{I_2^3} \|_{\ell^2}^{3/2} \ldots \big( \| b_{I_1^{m/2}} \|_{\ell^2}^4 \| b_{I_2^{m/2}} \|_{\ell^2}^2 \big)^{\frac{1}{2}} \ldots \big)^{\frac{1}{2}} \\
&\leq C_1^2 (C_\varepsilon \log(N_1)^{2+\varepsilon})^6 (D N_1^{-\frac{\alpha}{2}})^{\frac{1}{2^m}} \| b_{I_1} \|^4_{\ell^2} \| b_{I_2} \|_{\ell^2}^2.
\end{split}
\end{equation*}
In the ultimate estimate we used the disjointness between the intervals to carry out the sums. Recalling that $2^m \leq 4 \beta^{-1}$ we find
\begin{equation*}
\eqref{eq:DecouplingIteration} \leq C_1^2 (C_\varepsilon \log(N_1)^{2+\varepsilon})^6 D^{\frac{1}{\beta}} N^{-\frac{\alpha \beta }{8}} \| b_{I_1} \|^4_{\ell^2} \| b_{I_2} \|_{\ell^2}^2.
\end{equation*}
We summarize our findings as
\begin{equation*}
\| e^{it \Delta} f_1 (e^{it \Delta} f_2)^2 \|^2_{L^2([0,N^{-\alpha}] \times \T)} \leq (C_\varepsilon \log(N_1)^{2+\varepsilon})^6 D_\beta N_1^{-\frac{\alpha \beta}{8}} \| f_1 \|_{L^2}^2 \| f_2 \|_{L^2}^4 .
\end{equation*}
Coming back to \eqref{eq:OrthogonalReduction}, we obtain from \eqref{eq:InitialReduction1d} (recall that we suppose $\| f_i \|_{L^2} = 1$) and the estimate in the previous display:
\begin{equation*}
\begin{split}
\big\| \prod_{j=1}^3 e^{it \Delta} f_j \big\|^2_{L^2([0,N^{-\alpha}] \times \T)} &\leq C_\varepsilon \log(N_1)^{12+\varepsilon} D_\beta N_1^{-\frac{\alpha \beta}{8}} \\
&\leq C_{\varepsilon,\beta} \log(N_1)^{12+\varepsilon} N_1^{-\frac{\alpha \beta}{8}}.
\end{split}
\end{equation*}
In the last step we redefined $\varepsilon$ and $C_{\varepsilon,\beta}$ accordingly. The proof is complete.

$\hfill \Box$

\section{Refinement for weakly separated dyadic frequencies}
\label{section:LogarithmicRefinement}

In this section we prove Theorem \ref{thm:LogarithmicRefinementTrilinearStrichartz1d} as a variant of the previous argument. We shall be brief.

\begin{proof}[Proof~of~Theorem~\ref{thm:LogarithmicRefinementTrilinearStrichartz1d}]
%
By the reductions of Section \ref{subsection:InitialReductions} it suffices to estimate
\begin{equation}
\label{eq:BilinearDecouplingRefined}
\log(N)^{6+\varepsilon} \big( \int_{[0,N^{-\alpha}] \times \T} |e^{it \Delta} f_1|^2 |e^{it \Delta} f_2|^4 \big)^{\frac{1}{2}},
\end{equation}
where $\text{supp}(\hat{f}_i) \subseteq I_i$ an interval of length $N_3$ and $\text{dist}(I_1,I_2) \gtrsim N_1 =: N$. We rescale $x \to N x$, $t \to N^2 t$ to find
\begin{equation*}
\begin{split}
\eqref{eq:BilinearDecouplingRefined} &= N^{-3} \int_{[0,N^{2-\alpha}] \times N \T}  |e^{it \Delta} \tilde{f}_1|^2 |e^{it \Delta} \tilde{f}_2|^4 \\
&\lesssim N^{-3} N^{-(1-\alpha)} \int_{[0,N^{2-\alpha}] \times [0,N^{2-\alpha}]} |e^{it \Delta} \tilde{f}_1 |^2 |e^{it \Delta} \tilde{f}_2|^4.
\end{split}
\end{equation*}
In the above display we used periodicity which incurs a factor $N^{1-\alpha}$. We use continuous approximation like above
\begin{equation*}
\begin{split}
&\lesssim N^{-3} N^{-(1-\alpha)} \int_{B_{N^{2-\alpha}}} |e^{it \Delta} \tilde{f}_1|^2 |e^{it \Delta} \tilde{f}_2|^4 dt dx \\
&\lesssim N^{-3} N^{-(1-\alpha)} \int_{B_{N^{2-\alpha}}} |\mathcal{E} \tilde{f}_1|^2 |\mathcal{E} \tilde{f}_2|^4 w_{B_{N^{2-\alpha}}}^6 dx dt.
\end{split}
\end{equation*}
Now we use the decoupling iteration to find that the integral is dominated by
\begin{equation}
\label{eq:EstimateBRefined}
\lesssim \log^{13}(N) \big( \int |\mathcal{E} f_{1, \theta}|^2 |\mathcal{E} f_{2, \theta}|^4 w^6_{B_{N^{2-\alpha}}} dx dt \big)^{1/2^m}
\end{equation}
with $f_{i, \theta}$ supported in intervals of size $N^{-\frac{2-\alpha}{2}}$. $m$ describes the number of decoupling iterations, which is estimated by $m \leq \log_2 (10 \log \log (N))$. Indeed, the initial size of the intervals is $\leq \exp(\frac{-\log N}{\log \log N})$ after rescaling. Then it follows
\begin{equation*}
\exp \big( \frac{- \log N}{\log \log N} \big)^{2^m} \leq N^{-1}.
\end{equation*}

Reversing the continuous approximation and the scaling we find
\begin{equation*}
\int |\mathcal{E} f_{1, \theta}|^2 |\mathcal{E} f_{2, \theta}|^4 w_{B_{N^{2-\alpha}}}^6 dx dt \lesssim \int_{[0,N^{-\alpha}] \times \T} |e^{it \Delta} f_{1, \tilde{\theta}}|^2 |e^{it \Delta} f_{2, \tilde{\theta}}|^4 dx dt
\end{equation*}
with $\text{supp}(\hat{f}_{i, \tilde{\theta}}) \subseteq I_{i, \tilde{\theta}}$, which are intervals of length $N^{\frac{\alpha}{2}}$. The expression is finally estimated by short-time bilinear Strichartz estimates and Bernstein's inequality: 
\begin{equation}
\label{eq:EstimateCRefined}
\begin{split}
\int_{[0,N^{-\alpha}] \times \T} |e^{it \Delta} f_{1, \tilde{\theta}}|^2 |e^{it \Delta} f_{2, \tilde{\theta}}|^4 dx dt &\lesssim \| e^{it \Delta} f_{2, \tilde{\theta}} \|^2_{L^\infty_{t,x}} \| e^{it \Delta} f_{1, \tilde{\theta}} e^{it \Delta} f_{2, \tilde{\theta}} \|^2_{L^2_{t,x}([0,N^{-\alpha}] \times \T)} \\
&\lesssim N^{\frac{\alpha}{2}} N^{-\alpha} = N^{-\frac{\alpha}{2}}.
\end{split}
\end{equation}

Collecting estimates \eqref{eq:EstimateBRefined} and \eqref{eq:EstimateCRefined}, we have
\begin{equation*}
\int_{[0,N^{-\alpha}] \times \T} \prod_{i=1}^3 |e^{it \Delta} f_i |^2 dx dt \lesssim \log(N)^{13} N^{-\frac{\alpha}{2^{m+2}}} = \log(N)^{20} N^{-\frac{\alpha}{40 \log \log N}}.
\end{equation*}
Since $\log(N)^{20} N^{-\frac{\alpha}{40 \log \log N}} \to 0$ as $N \to \infty$, the proof is complete.

\end{proof}

\begin{remark}
We emphasize that it is not the separation property of the frequencies which is required for the above argument to prove a trilinear Strichartz estimate without derivative loss. This can always be recovered via a Whitney decomposition and affine rescaling (cf. \cite[Section~2]{GuoLiYungZorinKranich2021}). What is required is that the initial frequency intervals after rescaling are of size $O(\exp \big( \frac{-\log N}{\log \log N} \big))$. This keeps the number of decoupling iterations small enough for a favorable estimate.
\end{remark}

\section{Improved trilinear Strichartz estimates via transversality and decoupling iteration}

\label{section:TrilinearStrichartz2d}

\subsection{Set-up}

In the following we prove a trilinear analog in two dimensions. 
We estimate the expression
\begin{equation*}
\int_{[0,N_1^{-\alpha}] \times \T^2_\gamma} |e^{it \Delta} f_1 e^{it \Delta} f_2 e^{it \Delta} f_3 |^{\frac{4}{3}} dx dt
\end{equation*}
under support and transversality assumptions as stated in Theorem \ref{thm:TrilinearStrichartzEstimates2d}.

We recall the following quantitative trilinear Kakeya inequality due to Guth \cite{Guth2010} (see also \cite{CarberyValdimarsson2013}), which plays a crucial role in the argument:
\begin{theorem}
Let $(T_j^i)_j$ for $i=1,2,3$ denote three collections of tubes of size $\delta^{-1} \times \delta^{-1} \times \delta^{-2}$ such that for the long directions $\mathfrak{n}^i_{j_i}$ for tubes $T_{j_i}^i$ we have the uniform transversality condition:
\begin{equation*}
|\mathfrak{n}^1_{j_1} \wedge \mathfrak{n}^2_{j_2} \wedge \mathfrak{n}^3_{j_3}| \gtrsim \nu.
\end{equation*}
Let $(c_j^i)_{1 \leq j \leq M_i} \subseteq \R_{>0}$ with $i=1,2,3$ denote three collections of positive real numbers. Then the following holds:
\begin{equation}
\label{eq:QuantitativeTrilinearKakeya}
\int \big( \prod_{i=1}^3 \sum_{j=1}^{M_i} c_j^{(i)} \chi_{T_j^i} \big)^{\frac{1}{2}} \leq C \nu^{-\frac{1}{2}} \delta^{-3} \big( \prod_{i=1}^3 \sum_{j=1}^{M_i} c_j^{(i)} \big)^{\frac{1}{2}}.
\end{equation}
\end{theorem}

\subsection{Rescaling and continuous approximation}

We return to the expression \eqref{eq:MultilinearExpression2d}. The rescaling and approximation argument is analogous to Subsection \ref{subsection:ContinuousApproximation1d}. We shall be brief.
For the decoupling iteration we rescale space and time variables $x \to N_1 x$ and $t \to N_1^2 t$ and let $b_{i, N k} = \hat{f}_i(k)$ to find:
\begin{equation}
\label{eq:Rescaling2d}
\begin{split}
&\quad \int_{[0,N_1^{-\alpha}] \times \T_{\gamma}^2} \big| e^{it \Delta} f_1 e^{it \Delta} f_2 e^{it \Delta} f_3 \big|^{p/3} dx dt \\
&= N_1^{-4} \int_{[0,N_1^{2-\alpha}] \times N_1 \T_{\gamma}} \big| \prod_{i=1}^3 \sum_{\substack{k \in \Z^2 / N_1, \\ N_1 k \in \text{supp}(\hat{f}_i)}} e^{i(k \cdot x - t k^2)} b_{i,k} \big|^{p/3} dx dt.
\end{split}
\end{equation}

Let $\mathcal{E} \tilde{f}_i (x,t) = \int e^{i(x_1 \xi_1 + x_2 \xi_2 - t|\xi|^2)} \tilde{f}_i(\xi) d\xi$. We use periodicity in space and continuous approximation to obtain:
\begin{equation*}
\eqref{eq:Rescaling2d} \lesssim N_1^{-4} N_1^{-2(1-\alpha)} \int_{B_{N_1^{2-\alpha}}} \big| \mathcal{E} \tilde{f}_1 \mathcal{E} \tilde{f}_2 \mathcal{E} \tilde{f}_3 \big|^{p/3} dx dt.
\end{equation*}
Since this will simplify the argument, we consider the averaged expression for a decoupling iteration:
\begin{equation}
\label{eq:AveragedExpressionReductions}
\dashint_{B_{N_1^{2-\alpha}}} |\mathcal{E} \tilde{f}_1 \mathcal{E} \tilde{f}_2 \mathcal{E} \tilde{f}_3 \big|^{p/3} dx dt.
\end{equation}

We introduce the space-time weight $w_{B_{N_1^{2-\alpha}}}^{p/3}$ into the integral. This will be absorbed into the functions $\mathcal{E} \tilde{f}_i$ to find functions $\tilde{F}_i$, which are Fourier supported in the $N_1^{-(2-\alpha)}$ neighbourhood of the truncated paraboloid $\mathbb{P}^2 = \{ (\xi,-\xi^2) \, : \xi \in [-2,2]^2 \}$, which is denoted by $\mathcal{N}_{N_1^{-(2-\alpha)}} \mathbb{P}^2$. Moreover, the Fourier support of $\tilde{F}_i$ is restricted to space-frequency balls of size $N_3/N_1 \leq N_1^{-\frac{\beta}{1+\beta}}$. We have
\begin{equation}
\label{eq:AveragedExpressionReductionsII}
\eqref{eq:AveragedExpressionReductions} \lesssim \dashint \big| \tilde{F}_1 \tilde{F}_2 \tilde{F}_3 \big|^{p/3} dx dt.
\end{equation}



\subsection{The decoupling iteration}
\label{subsection:DecouplingIteration2d}
The qualitative argument follows \cite[Section~10.3]{Demeter2020}. The quantification rests upon two ingredients:
\begin{itemize}
\item the endpoint multilinear Kakeya inequality to make the dependence on the transversality precise,
\item keeping track of the number of decoupling iterations to obtain the precise gain from short-time bilinear Strichartz estimates.
\end{itemize}

\medskip
We introduce notations from \cite{Demeter2020}. 
Let $Q^r$ denote a cube with side length $2^r$ and $\mathbb{I}_{m}$ denote the collection of cubes tiling $[-2,2]^2$ with side length $2^{-m}$. For $I_{i,m} \in \mathbb{I}_{m}$ let $\mathcal{P}_{I_{i,m}}$ denote the Fourier projection to $I_{i,m} \times \R$. Let $L^t_{\#}(w_{Q^r})$ denote the averaged Lebesgue norm with a rapidly decreasing weight off $Q^r$:
\begin{equation*}
\| F \|^t_{L^t_{\#}(w_{Q^r})} = \frac{1}{|Q^r|} \int |F|^t w_{Q^r} dz.
\end{equation*}
With this notation, we have
\begin{equation*}
\text{rhs} \eqref{eq:AveragedExpressionReductionsII} \sim \| \tilde{F}_1 \tilde{F}_2 \tilde{F}_3\|^{p/3}_{L^{p/3}_{\#}(w_{Q_{N_1^{2-\alpha}}})}.
\end{equation*}

Let $\mathcal{Q}_{m,r}(Q^r)$ denote the partition of $Q^r$ with cubes of side-length $2^m$ and $m \leq r$.

\medskip

Let $\underline{F} = (F_i)_{i=1,2,3}$ be a family of functions with $\text{supp}(\hat{F}_i) \subseteq \mathcal{N}_{\delta^2}(\mathbb{P}^2)$ with $\delta \leq 2^{-m}$ and such that
\begin{equation}
\label{eq:TransversalityDecouplingIteration}
|(2\xi_1,1) \wedge (2\xi_2,1) \wedge (2\xi_3,1) | \geq \nu
\end{equation}
for $\xi_i \in \pi_{\R^2}(\text{supp}(\hat{F}_i))$.

In the present context we have
\begin{equation}
\label{eq:DeltaDecoupling}
\delta = N_1^{-(1-\frac{\alpha}{2})}.
\end{equation}

 We define for $\Delta_m \in \mathcal{Q}_{m,r}(Q^r)$
\begin{equation}
\label{eq:DQuantity}
D_t(m,\Delta_m,\underline{F},\delta) = \big[ \prod_{i=1}^3 \big( \sum_{I_{i,m} \in \mathbb{I}_{m}} \| \mathcal{P}_{I_{i,m}} F_i \|^2_{L^t_{\#}(w_{\Delta_m})} \big)^{\frac{1}{2}} \big]^{\frac{1}{3}},
\end{equation}
and
\begin{equation}
\label{eq:AQuantity}
A_p(m,Q^r,\underline{F},\delta) = \big( \frac{1}{|\mathcal{Q}_{m,r}|} \sum_{\Delta_m \in \mathcal{Q}_{m,r}(Q^r)} D_2(m,\Delta_m,\underline{F},\delta)^p \big)^{\frac{1}{p}}.
\end{equation}

As mentioned on \cite[p.~209]{Demeter2020}, using the locally constant property we can perceive $A_p$ as
\begin{equation*}
|Q^r| A_p^p(m,Q^r,\underline{F},\delta) \approx \int_{\R^3} \prod_{i=1}^3 \big( \sum_{I \in \mathbb{I}_m} |\mathcal{P}_I F_i|^2 \big)^{\frac{p}{3}} w_{Q^r}.
\end{equation*}
This representation suggests the possibility of exiting multilinearity via linear and bilinear Strichartz estimates.

\medskip

Moreover, we have the following representation for $D_p$
\begin{equation*}
\prod_{i=1}^3 \big( \sum_{I_{i,m} \in \mathbb{I}_m} \| \mathcal{P}_{I_{i,m}} F_i \|^2_{L^p_{\#}(w_{Q^r}} \big)^{\frac{1}{6}} = |Q^r|^{-\frac{1}{p}} \prod_{i=1}^3 \big( \sum_{I_{i,m}} \| \mathcal{P}_{I_{i,m}} F_i \|^2_{L^p(w_{Q^r})} \big)^{\frac{1}{6}}.
\end{equation*}

\medskip

Coming back to \eqref{eq:AveragedExpressionReductionsII}. We let
\begin{equation*}
\delta= N_1^{-(1-\frac{\alpha}{2})} \text{ and } 2^r = N_1^{2-\alpha}.
\end{equation*}
For notational convenience assume $r \in \N$ and $\delta^{-1} \in 2^{\N}$.

We find by H\"older's inequality
\begin{equation}
\label{eq:StartDecouplingIteration}
\dashint_{B_{N_1^{2-\alpha}}} \big| \tilde{F}_1 \tilde{F}_2 \tilde{F}_3 \big|^{4/3} \leq \frac{1}{|\mathcal{Q}_{m,r}|} \sum_{\Delta_m \in \mathcal{Q}_{m,r}} \| \tilde{F}_1 \|^{4/3}_{L^2_{\#}(w_{\Delta})} \| \tilde{F}_2 \|^{4/3}_{L^2_{\#}(w_\Delta)} \| \tilde{F}_3 \|^{4/3}_{L^2_{\#}(w_{\Delta})}.
\end{equation}

The key step is to lower the size of the Fourier support of $\tilde{F}_i$ in \eqref{eq:StartDecouplingIteration}, which is carried out in two steps. In the first step we show a ball inflation estimate, which is based on the trilinear Kakeya inequality. In the second step we use local $L^2$-orthogonality after suitable application of H\"older's inequality.

\begin{proposition}[Ball~inflation]
\label{prop:BallInflation}
Let $Q$ be a $\kappa^{-2}$-square, which is tiled with a collection $\mathcal{Q}$ of squares $\Delta$ with side length $\kappa^{-1}$ and $\text{supp}(\hat{F}_i) \subseteq \mathcal{N}_{\kappa^2}(\mathbb{P}^2)$, which satisfy \eqref{eq:TransversalityDecouplingIteration}.

For $\tilde{p} \geq 2$, the following estimate holds:
\begin{equation}
\label{eq:BallInflation} 
\begin{split}
&\frac{1}{|\mathcal{Q}|} \sum_{\Delta \in \mathcal{Q}} \prod_{i=1}^3 \big( \sum_{I_{i,m} \in \mathbb{I}_{m}(U_i)} \| \mathcal{P}_{I_{i,m}} F_i \|^2_{L^{\tilde{p}}_{\#}(w_{\Delta})} \big)^{\tilde{p}/4} \\
&\qquad \lesssim \nu^{-\frac{1}{2}} \log(\kappa^{-1})^c \prod_{i=1}^3 \big( \sum_{I_{i,m} \in \mathbb{I}_m(U_i)} \| \mathcal{P}_{I_{i,m}} F_i \|^2_{L^{\tilde{p}}_{\#}(w_Q)} \big)^{\tilde{p}/4}.
\end{split}
\end{equation}
\end{proposition}
For the qualitative argument we refer to \cite[Section~10.3]{Demeter2020}. By using the quantitative multilinear Kakeya estimate, we can obtain the dependence of decoupling on the transversality.

\begin{proof}

We use the locally constant properties for each square $I$ of size $\kappa$ to bring the multilinear Kakeya inequality into play. For $I \in \mathbb{I}_m(U_i)$ the Fourier support of $\mathcal{P}_I F_i$ is inside a rectangle $\mathfrak{r}_I$ of size comparable to $\kappa \times \kappa \times \kappa^2$. The long sides are given by $I$. We tile the cube $4Q$ with translates of the dual rectangle $T_I = \mathfrak{r}_I^o$, which family is denoted by $\mathcal{T}_I$. We define
\begin{equation*}
g_I^{(i)}(x) = \sup_{y \in 2 T_I(x)} \| \mathcal{P}_I F_i \|^{\tilde{p}}_{L^{\tilde{p}}_{\#}(w_{Q(y,\kappa^{-1})})},
\end{equation*}
which are constant on translates of $T_I$. It follows for $x \in \Delta$ that
\begin{equation}
\label{eq:TubeApproximation}
\| \mathcal{P}_I F_i \|^{\tilde{p}}_{L^{\tilde{p}}_{\#}(w_{\Delta})} \leq g_I^{(i)}(x)
\end{equation}
and (see \cite[Eq.~(11.15),~p.~251]{Demeter2020})
\begin{equation}
\label{eq:ReverseTubeApproximation}
\dashint_{4Q} g_I^{(j)}(x) dx \lesssim \| \mathcal{P}_I F_j \|^{\tilde{p}}_{L^{\tilde{p}}_{\#}(4Q)}.
\end{equation}
Moreover, by the definition of $g_I^{(i)}$ it follows
\begin{equation*}
g_I^{(i)}(x) = \sum_{T \in \mathcal{T}_I} c_T 1_T.
\end{equation*}
We let $g_i(x) = \sum_{I \in \mathbb{I}_m(U_i)} g_I^{(i)}(x)$.

By dyadic pigeonholing we can suppose that $\| \mathcal{P}_{I} F_i \|_{L^{\tilde{p}}_{\#}(w_{\Delta})}$ have comparable size. This incurs a factor $\log(\kappa^{-1})^c$. We assume that we have $M_i$ squares $J_i$ with $\| \mathcal{P}_{J_i} F_i \|_{L^{\tilde{p}}_{\#}(w_Q)}$ of comparable size. By H\"older's inequality we find
\begin{equation*}
\begin{split}
\text{lhs} \eqref{eq:BallInflation} &\lesssim \log(\kappa^{-1})^c \big( \prod_{i=1}^3 M_i^{\frac{1}{2}-\frac{1}{\tilde{p}}} \big)^{\frac{\tilde{p}}{2}} \frac{1}{|\mathcal{Q}|} \sum_{\Delta \in \mathcal{Q}} \prod_{i=1}^3 \big( \sum_{I \in \mathbb{I}_m(U_i)} \| \mathcal{P}_I F_i \|^{\tilde{p}}_{L^{\tilde{p}}_{\#}(w_{\Delta})} \big)^{\frac{1}{2}} \\
&\lesssim \log(\kappa^{-1})^c \big( \prod_{i=1}^3 M_i^{\frac{1}{2}-\frac{1}{\tilde{p}}} \big)^{\frac{\tilde{p}}{2}} \dashint_Q \big( \prod_{i=1}^3 g_i \big)^{\frac{1}{2}}.
\end{split}
\end{equation*}


The transversality assumption \eqref{eq:TransversalityDecouplingIteration} yields the quantitative transversality for the families of tubes.
An application of the endpoint trilinear Kakeya estimate \eqref{eq:QuantitativeTrilinearKakeya} gives
\begin{equation*}
\int \big( \prod_{i=1}^3 g_i \big)^{\frac{1}{3}} \lesssim \nu^{-\frac{1}{2}} \delta^{-3} \big( \prod_{j=1}^3 \sum_{T \in \mathcal{T}_j} c_T \int \chi_T / (\delta^{-4}) \big)^{\frac{1}{2}}.
\end{equation*}
Hence, by \eqref{eq:ReverseTubeApproximation}
\begin{equation*}
\frac{1}{|Q|} \int \big( \prod_{i=1}^3 g_i \big)^{\frac{1}{2}} \lesssim \big( \prod_{j=1}^3 \frac{1}{|Q|} \int \sum_{T \in \mathcal{T}_j} c_T \chi_T \big)^{\frac{1}{2}} \lesssim \big( \prod_{j=1}^3 \sum_{J_j} \| \mathcal{P}_{J_j} F_j \|^{\tilde{p}}_{L^{\tilde{p}}_{\#}(w_Q)} \big)^{\frac{1}{2}}.
\end{equation*}

Since we assume that $\| \mathcal{P}_I F_i \|_{L^{\tilde{p}}_{\#}(w_\Delta)}$ have comparable size, it follows
\begin{equation*}
\big( \prod_{i=1}^3 M_i^{\frac{1}{2}-\frac{1}{\tilde{p}}} \big)^{\frac{\tilde{p}}{2}} \big( \prod_{j=1}^3 \sum_{J_i} \| \mathcal{P}_{J_i} F_j \|^{\tilde{p}}_{L^{\tilde{p}}_{\#}(w_Q)} \big)^{\frac{1}{2}} \lesssim \big[ \prod_{j=1}^3 \big( \sum_{J_i} \| \mathcal{P}_{J_i} F_j \|^2_{L^{\tilde{p}}_{\#}(w_Q)} \big)^{\frac{1}{2}} \big]^{\frac{\tilde{p}}{2}}.
\end{equation*}
The proof is complete.

\end{proof}

By an application of H\"older's inequality and relating $p$ and $\tilde{p}$ we find the following (see \cite[Proposition~10.23]{Demeter2020}):
\begin{proposition}[Multiscale~inequality]
\label{prop:MultiscaleInequality}
Let $3 \leq p \leq 4$ and
\begin{equation*}
\frac{3}{2p} = \frac{\kappa_p}{2} + \frac{1- \kappa_p}{p}.
\end{equation*}
Then there is $c>0$ such that the following holds:
\begin{equation}
\label{eq:Multiscale3d}
A_p(m,Q^{2m},\underline{F},\delta) \lesssim \nu^{-\frac{1}{2p}} \log(\delta^{-1})^c A_p(2m,Q^{2m},\underline{F},\delta)^{\kappa_p} D_p(m,Q^{2m},\underline{F},\delta)^{1-\kappa_p}.
\end{equation}
\end{proposition}
\begin{proof}
For $p \in [3,4]$ we have that $\tilde{p} = \frac{2p}{3} \geq 2$. This allows us to use H\"older's inequality:
\begin{equation*}
\frac{1}{|\mathcal{Q}|} \sum_{\Delta \in \mathcal{Q}} \prod_{i=1}^3 \big( \sum_{I \in \mathbb{I}_m(U_i)} \| \mathcal{P}_I F_i \|^2_{L^2_{\#}(w_\Delta)} \big)^{\frac{p}{6}} \lesssim \frac{1}{|\mathcal{Q}|} \sum_{\Delta \in \mathcal{Q}} \prod_{i=1}^3 \big( \sum_{I \in \mathbb{I}_m(U_1)} \| \mathcal{P}_I F_i \|^2_{L^{\tilde{p}}_{\#}(w_\Delta)} \big)^{\frac{\tilde{p}}{4}}.
\end{equation*}
Now we can apply Proposition \ref{prop:BallInflation} to find
\begin{equation*}
\frac{1}{|\mathcal{Q}|} \sum_{\Delta \in \mathcal{Q}} \prod_{i=1}^3 \big( \sum_{I \in \mathbb{I}_m(U_i)} \| \mathcal{P}_I F_i \|^2_{L^2_{\#}(w_\Delta)} \big)^{\frac{p}{6}} \lesssim \nu^{-\frac{1}{2}} \log(\delta^{-1})^c
\prod_{i=1}^3 \big( \sum_{I \in \mathbb{I}_m(U_i)} \| \mathcal{P}_I F_i \|^2_{L^{\frac{2p}{3}}_{\#}(w_Q)} \big)^{\frac{p}{6}}.
\end{equation*}
An application of H\"older's inequality gives
\begin{equation*}
\begin{split}
\prod_{i=1}^3 \big( \sum_{I \in \mathbb{I}_m(U_i)} \| \mathcal{P}_I F_i \|^2_{L^{\frac{2p}{3}}_{\#}(w_Q)} \big)^{\frac{p}{6}} &\lesssim \prod_{i=1}^3 \big( \sum_{I \in \mathbb{I}_m(U_i)} \| \mathcal{P}_I F_i \|^2_{L^{2}_{\#}(w_Q)} \big)^{\frac{\kappa_p p}{6}} \\
&\qquad \times \prod_{i=1}^3 \big( \sum_{I \in \mathbb{I}_m(U_i)} \| \mathcal{P}_I F_i \|^2_{L^{p}_{\#}(w_Q)} \big)^{\frac{(1-\kappa_p) p}{6}}.
\end{split}
\end{equation*}

Finally, we can use local $L^2$-orthogonality (see \cite[Proposition~10.20]{Demeter2020}) to find
\begin{equation*}
\sum_{I \in \mathbb{I}_m(U_i)} \| \mathcal{P}_I F_i \|^2_{L^{2}_{\#}(w_Q)} \lesssim \sum_{I \in \mathbb{I}_{2m}(U_i)} \| \mathcal{P}_I F_i \|^2_{L^2_{\#}(w_Q)}.
\end{equation*}
The proof is complete.
\end{proof}

We obtain the following from iterating the multiscale inequality: Let 
\begin{equation}
\label{eq:NotationsDecouplingIteration}
2^{-\underline{m}} = \frac{N_3}{N_1} \leq N_1^{-\frac{\beta}{1+\beta}} \text{ and } 2^{\overline{m}} = N_1^{2-\alpha} = \delta^{-2}.
\end{equation}
 Then by \eqref{eq:StartDecouplingIteration} and Proposition \ref{prop:MultiscaleInequality},
\begin{equation}
\label{eq:MultiscaleExpressionI}
\begin{split}
\frac{1}{|Q^{\overline{m}}|} \int |F_1 F_2 F_3 |^{p/3} dx dt &\lesssim A_p^p(\underline{m},Q^{\overline{m}},\underline{F},\delta) \\
&\lesssim \nu^{-\frac{1}{2}} \log(N_1)^c A_p^p (2 \underline{m}, Q^{\overline{m}}, \underline{F}, \delta)^{\kappa_p} D_p^p(\underline{m},Q^{\overline{m}},\underline{F},\delta)^{1-\kappa_p}.
\end{split}
\end{equation}
Iterating the multiscale inequality $n$ times, we find
\begin{equation}
\label{eq:MultiscaleExpressionII}
\begin{split}
A_p^p(\underline{m},Q^{\overline{m}},\underline{F},\delta) &\lesssim \nu^{-\frac{1}{2}} \log(N_1)^c \big( \prod_{\ell=1}^{n-1} \nu^{-\frac{\kappa_p^\ell}{2}} \log(N_1)^{c \kappa_p^{\ell}} \big) \\
&\quad \times A_p^p(2^n \underline{m}, Q^{\overline{m}}, \underline{F},\delta)^{\kappa_p^n} \prod_{\ell=1}^{n} D_p^p(2^{\ell -1} \underline{m}, Q^{\overline{m}}, \underline{F},\delta)^{\kappa_p^{\ell - 1} (1-\kappa_p)}.
\end{split}
\end{equation}
The number of iterations of the multiscale inequality is at most
\begin{equation}
\label{eq:MaxIterations}
n^* = \lceil \log_2(\beta^{-1} + 1) \rceil
\end{equation}
because
\begin{equation*}
N_1^{-\frac{\beta}{1+\beta} 2^{n^*}} \leq N_1^{-1}.
\end{equation*}

\subsection{Exiting the decoupling iteration with linear and bilinear Strichartz estimates}

We exit the decoupling iteration similar to the one-dimensional case by abandoning all multilinearity and estimate the resulting expressions via linear and bilinear Strichartz estimates. 

\subsubsection{Linear and bilinear Strichartz estimates}

We recall the following linear Strichartz estimates due to Bourgain--Demeter \cite{BourgainDemeter2015} combined with Galilean invariance.
\begin{theorem}[{\cite{BourgainDemeter2015}}]
\label{thm:LinearStrichartz}
Let $S \subseteq B(\xi^*,M) \subseteq \R^2$ and $f_S = \sum_{k \in S} e^{i kx} a_k: \T_\gamma^2 \to \C$. Then the following estimate holds:
\begin{equation*}
\| e^{it \Delta} f_S \|_{L_t^4([0,1],L_x^4(\T_\gamma^2))} \lesssim_\varepsilon M^\varepsilon \| a_S \|_{\ell^2}.
\end{equation*}
\end{theorem}
The estimate with logarithmically sharpened constant by lattice point counting 
\begin{equation*}
\| e^{it \Delta} f_S \|_{L_{t,x}^4([0,1] \times \T^2)} \lesssim \exp \big( \frac{c \log M}{\log \log M} \big) \| f_S \|_{L^2}
\end{equation*}
is due to Bourgain \cite{Bourgain1993A}. Very recently, Herr--Kwak \cite{HerrKwak2023} reported the estimate with sharp constant $\log(M)^{\frac{1}{4}}$ for $M \geq 2$.

\medskip

The bilinear Strichartz estimates due to Hani \cite{Hani2012} in two dimensions are given as follows:
\begin{theorem}
\label{thm:BilinearStrichartz2d}
Let $N_2 \ll N_1$ and $S \subseteq B(\xi^*,M) \subseteq B(0,N_2)$. Then it holds
\begin{equation*}
\| e^{it \Delta} P_{N_1} f_1 e^{it \Delta} P_{N_2} P_S f_2 \|_{L^2_{t,x}([0,N_1^{-1}] \times \T_\gamma^2)} \lesssim \big( \frac{M}{N_1} \big)^{\frac{1}{2}} \| P_{N_1} f_1 \|_{L^2} \| P_{N_2} f_2 \|_{L^2}.
\end{equation*}
\end{theorem}

We shall use the below estimate, which follows from piecing together intervals of length $N_1^{-1}$ and using Galilean invariance:
\begin{corollary}
Let $0<\alpha \leq 1$. Under the assumptions of Theorem \ref{thm:BilinearStrichartz2d}, the following estimate holds:
\begin{equation}
\label{eq:DeltaTimeInterval}
\| e^{it \Delta} P_{N_1} f_1 e^{it \Delta} P_{N_2} P_S f_2 \|_{L^2_{t,x}([0,N_1^{-\alpha}] \times \T_\gamma^2)} \lesssim \frac{M^{\frac{1}{2}}}{N_1^{\frac{\alpha}{2}}} \| P_{N_1} f_1 \|_{L^2} \| P_{N_2} P_S f_2 \|_{L^2_x}.
\end{equation}
\end{corollary}

\subsubsection{Linear estimate}

When exiting the decoupling iteration, we reverse the scaling and continuous approximation. Let $I \in \mathbb{I}_{m_1}(B(0,1))$ and $\tilde{I} = N_1 I$. We indicate the Fourier support by subindices. Then it holds:
\begin{equation*}
\begin{split}
\| \mathcal{P}_I F \|^4_{L^4(w_{N_1^{2-\alpha}})} &\lesssim N_1^{4} N_1^{2(1-\alpha)} \int_{[0,N_1^{-\alpha}] \times \T_\gamma^2} |e^{it \Delta} a_{\tilde{I}} |^4 dx dt \\
&\lesssim_\varepsilon N_1^4 N_1^{2(1-\alpha)} N_1^\varepsilon \| a_{\tilde{I}} \|_{\ell^2}^4.
\end{split}
\end{equation*}
This provides us with the following estimate for $D_4$: 

\begin{lemma}
\label{lem:LinearExitEstimate2d}
With notations from \eqref{eq:NotationsDecouplingIteration}, the following estimate holds:
\begin{equation}
\label{eq:LinearExitEstimate}
\begin{split}
D_4^4(m,Q,\underline{F},\delta) &= \frac{1}{|Q|} \big[ \prod_{i=1}^3 \big( \sum_{I_{i,m} \in \mathbb{I}_m} \| P_{I_{i,m}} F_i \|^2_{L^4(w_{Q^r})} \big)^{\frac{1}{2}} \big]^{\frac{4}{3}} \\
&\lesssim_\varepsilon \frac{N_1^{4+2(1-\alpha)}}{|Q|} N_1^\varepsilon \big( \prod_{i=1}^3 \| f_i \|_{L^2} \big)^{\frac{4}{3}}.
\end{split}
\end{equation}
\end{lemma}

\subsubsection{Bilinear estimates}

For $A^4_4$ we expect by the locally constant property
\begin{equation}
\label{eq:HeuristicLocallyConstant}
\begin{split}
&\quad \frac{1}{|\mathcal{Q}_{m,r}|} \sum_{\Delta_m \in \mathcal{Q}_{m,r}(Q^r)} \big( \prod_{i=1}^3 \big( \sum_{I_{i,m}} \| \mathcal{P}_{I_{i,m}} F_i \|^2_{L^2_{\#}(w_{\Delta_m})} \big)^{\frac{1}{2}} \big)^{\frac{p}{3}} \\
&\approx \dashint_Q \prod_{i=1}^3 \big( \sum_I |\mathcal{P}_I F_i|^2 \big)^{\frac{p}{6}} \\
&\lesssim \frac{1}{|Q|} \int_Q \big( \sum_{I_1} |\mathcal{P}_{I_1} F_1|^2 \big)^{\frac{2}{3}} \big( \sum_{I_2} |\mathcal{P}_{I_2} F_2|^2 \big)^{\frac{2}{3}} \big( \sum_{I_3} |\mathcal{P}_{I_3} F_3|^2 \big)^{\frac{2}{3}}.
\end{split}
\end{equation}

Next, we use H\"older's inequality to find
\begin{equation*}
\begin{split}
&\quad \big( \int_Q \sum_{I_1} |\mathcal{P}_{I_1} F_1|^2 \sum_{I_3} |\mathcal{P}_{I_3} F_3|^2 \big)^{\frac{2}{3}} \big( \int_Q \big( \sum_I |\mathcal{P}_I F_2|^2 \big)^2 \big)^{\frac{1}{3}} \\
&= \big( \sum_{I_1,I_3} \int_Q |\mathcal{P}_{I_1} F_1 \mathcal{P}_{I_3} F_3|^2 \big)^{\frac{2}{3}} \big( \int_Q \big( \sum_I |\mathcal{P}_I F_2|^2 \big)^{4/2} \big)^{\frac{4}{12}}.
\end{split}
\end{equation*}

After inverting the scaling and reversing the continuous approximation, the above expression can be estimated by short-time bilinear Strichartz estimate and linear Strichartz estimates as recorded above.

\begin{proposition}
\label{prop:EstimateApp}
With above notations, the following estimate holds:
\begin{equation}
\label{eq:ExitEstimateApp}
A_4^4(\underline{m},Q^{\overline{m}},\underline{F},\delta) \lesssim_\varepsilon \frac{1}{|Q|} N_1^4 N_1^{2(1-\alpha)} N_1^{-\frac{\alpha}{3}} N_1^\varepsilon \big( \prod_{i=1}^3 \| f_i \|_2 \big)^{\frac{4}{3}}.
\end{equation}
\end{proposition}

\begin{proof} To make $\approx$ in \eqref{eq:HeuristicLocallyConstant} precise, we use the uncertainty principle. With the Fourier transform of $\mathcal{P}_{I_{i,m}} F_i$ being supported in a $\delta \times \delta \times \delta$-box, by Lemma \ref{lem:Uncertainty} we find
\begin{equation*}
|\mathcal{P}_{I_{i,m}} F_i|^2 (x) \lesssim |\mathcal{P}_{I_{i,m}} F_i|^2 * \phi_{I} (x)
\end{equation*}
with $\phi_I$ rapidly decaying off $\delta^{-1} \times \delta^{-1} \times \delta^{-1}$ and locally constant at scale $\delta^{-1}$.

 We have for any $x \in Q_m$ by locally constant property of $\phi_{I}$
\begin{equation*}
\begin{split}
\| \mathcal{P}_{I_{im}} F_i \|^2_{L^2_{\#}(w_{Q_m})} &= \frac{1}{|Q_m|} \int \big( |\mathcal{P}_{I_{im}} F_i|^2 * \phi_{I} \big) w_{Q_m} \\
&= \sum_{k_i \in \Z^3} \frac{1}{|Q_m|} \int_{Q_m + k_i \delta^{-1}} \big( |\mathcal{P}_{I_{i,m}} F_i|^2 * \phi_I \big) w_{Q_m} \\
&\lesssim \sum_{k_i \in \Z^3} (|\mathcal{P}_{I_{i,m}} F_i|^2 * \phi_{I_{i,m}} )(x+k_i \delta^{-1}) \omega(k_i).
\end{split}
\end{equation*}
We have discretized $w_{Q_m}$ by $\omega(k) \lesssim (1+|k|)^{-100}$.
This implies the estimate:
\begin{equation*}
\begin{split}
&\quad \frac{1}{|\mathcal{Q}_{m,r}|} \sum_{Q^m \in \mathcal{Q}_{m,r}} \dashint_{Q^m} \big( \prod_{i=1}^3 \big( \sum_{I_{i,m}} \| \mathcal{P}_{I_{i,m}} F_i \|^2_{L^2_{\#}(w_{Q^m})} \big)^{\frac{1}{2}} \big)^{\frac{p}{3}} \\
&\lesssim \dashint_Q \big( \prod_{i=1}^3 \big( \sum_{I_{i,m}} \sum_{k_i \in \Z^3} \big( |\mathcal{P}_{I_{i,m}} F_i|^2 * \phi_{I_{i,m}} \big) (x+k_i \delta^{-1}) \omega(k_i) \big)^{\frac{1}{2}} \big)^{\frac{p}{3}}.
\end{split}
\end{equation*}
We use H\"older's inequality to obtain
\begin{equation*}
\begin{split}
&\quad \int_Q \prod_{j=1}^3 \big( \sum_{I_j,k_j} |\mathcal{P}_{I_j} F_j|^2 * \phi_{I}(z+k_j \delta^{-1}) \omega(k_j) \big)^{\frac{2}{3}} \\
 &\lesssim \big( \sum_{\substack{ I_1,k_1, \\ I_3,k_3}} \int_Q |\mathcal{P}_{I_1} F_1|^2 * \phi_{I_1}(z+k_1 \delta^{-1}) \omega(k_1) \; |\mathcal{P}_{I_3} F_3|^2*\phi_{I_2}(z+k_3 \delta^{-1}) \omega(k_3) \big)^{\frac{2}{3}} \\
 &\quad \times \big( \int_Q \big( \sum_{I_3,k_3} |\mathcal{P}_{I_3} F_3|^2 * \phi_{I_3}(z+k_3 \delta^{-1}) \omega(k_3) \big)^2 \big)^{\frac{1}{3}}.
 \end{split}
\end{equation*}
The proof is concluded with Lemmas \ref{lem:EstimateApp1} and \ref{lem:EstimateApp2} below.
\end{proof}

\begin{lemma}
\label{lem:EstimateApp1}
The following estimate holds:
\begin{equation}
\label{eq:ExitBilinearStrichartzEstimate2d}
\begin{split}
&\quad \int_Q |\mathcal{P}_{I_1} F_1|^2 * \phi_{I_1}(z+k_1 \delta^{-1}) |\mathcal{P}_{I_2} F_2|^2 * \phi_{I_2}(z+k_2 \delta^{-1}) dz \\
&\lesssim N_1^4 N_1^{2(1-\alpha)} N_1^{-\frac{\alpha}{2}} \| b_{1,I_1} \|^2_{\ell^2} \| b_{2,I_2} \|^2_{\ell^2}.
\end{split}
\end{equation}
\end{lemma}
The proof is similar to the proof of Lemma \ref{lem:BilinearExitLemma}. Here we need the separation condition \eqref{eq:SeparationCondition} to use the bilinear Strichartz estimate.
\begin{proof}
We write $\mathcal{P}_{I_j} F_j = \mathcal{E} b_{j,I_j} \cdot w_{B_{N_1^{2-\alpha}}}$ and estimate
\begin{equation*}
\begin{split}
&\quad \int_Q |\mathcal{P}_{I_1} F_1|^2 * \phi_{I_1}(x+k_1 \delta^{-1}) |\mathcal{P}_{I_2} F_2|^2 * \phi_{I_2}(x+k_2 \delta^{-1}) dx \\
&\lesssim \iint \int_Q |\mathcal{E} b_{1,I_1}|^2 (z+k_1 \delta^{-1} - y_1) |\mathcal{E} b_{2,I_2}|^2 (z+k_2 \delta^{-1} - y_2) dz \; \phi_I(y_1) dy_1 \; \phi_I(y_2) dy_2.
\end{split}
\end{equation*}
We prove an estimate for the space-time integral over $Q$ which is independent of the shifts $k_i \delta^{-1} - y_i$ by absorbing the dilation into $b_{I_1}$:
\begin{equation*}
\int_Q |\mathcal{E} \tilde{f}_{1,I_1}|^2 (z+k_1 \delta^{-1} - y_1) |\mathcal{E} \tilde{f}_{2,I_2}|^2 (z+k_2 \delta^{-1} - y_2) dz = \int_Q |\mathcal{E} f'_{1,I_1}|^2 (z) |\mathcal{E} f'_{2,I_2}|^2 (z) dz.
\end{equation*}
Now we can reverse the change of variables and the continuous approximation to find
\begin{equation*}
\lesssim N_1^4 N_1^{2(1-\alpha)} \int_{[0,N_1^{-\alpha}] \times \T^2_\gamma} |e^{it \Delta} f'_{1,I_1}|^2 (x,t) |e^{it \Delta} f'_{2,I_2}|^2(x,t) dx dt.
\end{equation*}
The functions $f_i$ are Fourier supported in balls of size $N_1^{\alpha/2}$ and separated at a distance comparable to $N_1$. We can estimate this by the short-time bilinear Strichartz estimate to find
\begin{equation*}
\int_{[0,N_1^{-\alpha}] \times \T^2_\gamma} |e^{it \Delta} f'_{I_1} |^2 |e^{it \Delta} f'_{I_2}|^2 dx dt \lesssim N_1^{-\frac{\alpha}{2}} \| b_{1,I_1} \|^2_{\ell^2} \| b_{2,I_2} \|^2_{\ell^2}.
\end{equation*}
\end{proof}

We turn to the estimate for the linear part:
\begin{lemma}
\label{lem:EstimateApp2}
The following estimate holds:
\begin{equation*}
\int_Q \big( \sum_{I_2,k_2} |\mathcal{P}_{I_2} F_2|^2 * \phi_{I_2} (x+k_2 \delta^{-1}) \omega(k_2) \big)^2 \lesssim_\varepsilon N_1^4 N_1^{2(1-\alpha)} N_1^\varepsilon \| f_2 \|_{L^2}^4.
\end{equation*}
\end{lemma}
\begin{proof}
We write the left hand-side as
\begin{equation*}
\begin{split}
&\quad \int_Q \sum_{I_{21},k_{21}} |\mathcal{P}_{I_{21}} F_2|^2 * \phi_{I_{21}}(x+k_{21} \delta^{-1}) \omega(k_{21}) \sum_{I_{22},k_{22}} |\mathcal{P}_{I_{22}} F_2|^2 * \phi_{I_{22}}(x+k_{22} \delta^{-1}) \omega(k_{22}) \\
&= \sum_{\substack{I_{21},k_{21}, \\ I_{22},k_{22}}} \omega(k_{21}) \omega(k_{22}) \int_Q \big( |\mathcal{P}_{I_{21}} F_2|^2 * \phi_{21}(x+k_{21} \delta^{-1}) \big) \big( |\mathcal{P}_{I_{22}} F_2|^2 * \phi_{I_{22}}(x+k_{22} \delta^{-1}) \big).
\end{split}
\end{equation*}
We prove the estimate
\begin{equation*}
\begin{split}
&\quad \int_Q |\mathcal{P}_{I_{21}} F_2|^2 * \phi_{I_{21}}(x+k_{21} \delta^{-1}) |\mathcal{P}_{I_{22}} F_2|^2 * \phi_{I_{22}}(x+k_{22} \delta^{-1}) \\ 
&\lesssim_\varepsilon N_1^4 N_1^{2(1-\alpha)} N_1^\varepsilon \| b_{2,I_{21}} \|_{\ell^2}^2 \| b_{2,I_{22}} \|^2_{\ell^2},
\end{split}
\end{equation*}
which is independent of $k_{21}$, $k_{22}$.

The argument follows along the lines of the proof of \eqref{eq:ExitBilinearStrichartzEstimate2d}:
\begin{equation*}
\iint \int_Q |\mathcal{P}_{I_{21}} F_2|^2 (x+k_{21} \delta^{-1} - y_1) |\mathcal{P}_{I_{22}} F_2|^2(x+k_{22} \delta^{-1} - y_2) dz \, \phi_{I_{21}}(y_1) \phi_{I_{22}} (y_2) dy_1 dy_2
\end{equation*}
and turn to an estimate of the integral independent of $k_{21}$, $k_{22}$, $y_1$, $y_2$. Write
\begin{equation*}
\mathcal{P}_{I_{2j}} F_2 = \mathcal{E} \tilde{f}_{I_{2j}} \cdot w_{B_{N_1^{2-\alpha}}}.
\end{equation*}
We can estimate the weight with a positive constant. Like above we can absorb the shift $k_{2i} \delta^{-1} - y_i$ into the argument $\tilde{f}_{I_i}$ and reversing the scaling and the continuous approximation we obtain
\begin{equation*}
\begin{split}
&\quad \int_Q |\mathcal{P}_{I_{21}} F_2|^2(z+k_{21} \delta^{-1} - y_1) |\mathcal{P}_{I_{22}} F_2|^2(z+k_{22} \delta^{-1} - y_2) dz \\
&\lesssim N_1^4 N_1^{2(1-\alpha)} \int_{[0,N_1^{-\alpha}] \times \T^2_\gamma} |e^{it \Delta} f'_{I_{21}}|^2 |e^{it \Delta} f'_{I_{22}}|^2 dx dt.
\end{split}
\end{equation*}
This is estimated by H\"older's inequality and the $L^4_{t,x}$-Strichartz estimates:
\begin{equation*}
\int_{[0,N_1^{-\alpha}] \times \T^2_\gamma} |e^{it \Delta} \tilde{f}_{I_{21}}|^2 |e^{it \Delta} \tilde{f}_{I_{22}}|^2 dt dx \lesssim_\varepsilon N_1^\varepsilon \| b_{I_{21}} \|^2_{\ell^2} \| b_{I_{22}} \|^2_{\ell^2}.
\end{equation*}

\end{proof}


\subsection{Conclusion}

Now we can estimate \eqref{eq:MultiscaleExpressionI} and \eqref{eq:MultiscaleExpressionII} in terms of the original functions $f_i$. Recall that the number of iterations is estimated by
\begin{equation*}
n^* = \lceil \log_2(\beta^{-1}+1) \rceil.
\end{equation*}

For $p=4$ we obtain as consequence of $\kappa_4 = \frac{1}{2}$:
\begin{equation*}
\nu^{-\frac{1}{2}} \log(N_1)^c \big( \prod_{\ell=1}^{n-1} \nu^{-\frac{\kappa_p^\ell}{2}} \log(N_1)^{c \kappa_p^{\ell}} \big) \leq \nu^{-1+2^{-n}} \log(N_1)^{2c}.
\end{equation*}
We find by Lemma \ref{lem:LinearExitEstimate2d} and Proposition \ref{prop:EstimateApp} that
\begin{equation}
\label{eq:MultiScaleEstimateI}
\begin{split}
&\quad A_p^p(2^n \underline{m},Q^{\overline{m}},\underline{F},\delta)^{\kappa_p^n} \prod_{\ell=1}^{n} D_p^p(2^{\ell-1} \underline{m}, Q^{\overline{m}}, \underline{F}, \delta)^{\kappa_p^{\ell}} \\
&\lesssim_\varepsilon \frac{1}{|Q|} N_1^4 N_1^{2(1-\alpha)} N_1^{-\frac{\alpha}{3} 2^{-n}} N_1^\varepsilon \big( \prod_{i=1}^3 \| f_i \|_2 \big)^{\frac{4}{3}}.
\end{split}
\end{equation}

By the bound on the number of iterations, we find
\begin{equation*}
\nu^{-1+2^{-n}} = \nu^{-1 + \frac{\beta}{2(1+\beta)}}.
\end{equation*}
Consequently, it holds
\begin{equation*}
\eqref{eq:MultiScaleEstimateI} \lesssim_\varepsilon \frac{1}{|Q|} N_1^4 N_1^{2(1-\alpha)} N_1^{-\frac{\alpha \beta}{12}} N_1^\varepsilon \big( \prod_{i=1}^3 \| f_i \|_2 \big)^{\frac{4}{3}}.
\end{equation*}

Taking the above estimates together we find
\begin{equation*}
\int_{[0,N_1^{-\delta}] \times \T_\gamma^2} \big| e^{it \Delta} f_1 e^{it \Delta} f_2 e^{it \Delta} f_3 \big|^{\frac{4}{3}} dx dt \lesssim_\varepsilon \nu^{-1 + \frac{\beta}{2(1+\beta)}}  N_1^{- \frac{\alpha \beta}{12} +\varepsilon} \big( \prod_{i=1}^3 \| f_i \|_2 \big)^{\frac{4}{3}}.
\end{equation*}

The proof of Theorem \ref{thm:TrilinearStrichartzEstimates2d} is complete.
$\hfill \Box$

\subsection{Comparison with $L^3$-estimates}

In this subsection we compare the obtained estimates to an $L^3$-estimate with sharp dependence on transversality.
\subsubsection{The trilinear restriction estimate}
We have the following trilinear $L^3_{t,x}$-estimate with sharp dependence on transversality, which is due to Ramos \cite{Ramos2018}:
\begin{theorem}
Let $R \gg 1$, $\nu \in (0,1]$, and $f_i: \R^2 \to \C$ with $\text{supp}(f_i) \subseteq B_2(0,1) \subseteq \R^2$ with $|\mathfrak{n}(\xi_1) \wedge \mathfrak{n}(\xi_2) \wedge \mathfrak{n}(\xi_3)| \gtrsim \nu$ for $\xi_i \in \text{supp}(f_i)$. Then there is $c>0$ such that the following estimate holds:
\begin{equation*}
\int_{B_3(0,R)} \big| \prod_{i=1}^3 \mathcal{E} f_i \big| \lesssim \log(R)^c \nu^{-\frac{1}{2}} \prod_{i=1}^3 \| f_i \|_{L^2(\R^2)}.
\end{equation*}
\end{theorem}

This yields the following decoupling estimate as observed by Bourgain in \cite{Bourgain2013} (see \cite[Section~9.2]{Demeter2020} for further context):
\begin{corollary}
\label{cor:DecouplingL3Transversality}
Let $f_i \in \mathcal{S}(\R^3)$, $i=1,2,3$, $B^{(i)} \subseteq B_2(0,1)$ with $\text{supp}(\hat{f}_i) \subseteq \mathcal{N}_{R^{-1}}(\{(\xi,-|\xi|^2 ) : \xi \in B^{(i)} \})$, and it holds for $\xi_i \in B^{(i)}$
\begin{equation*}
|\mathfrak{n}(\xi_1) \wedge \mathfrak{n}(\xi_2) \wedge \mathfrak{n}(\xi_3) | \gtrsim \nu.
\end{equation*}
Then the following estimate holds:
\begin{equation}
\label{eq:L3DecouplingTransversality}
\int_{B_3(0,R)} \big| \prod_{i=1}^3 f_i \big| \lesssim \log(R)^c \nu^{-\frac{1}{2}} \prod_{i=1}^3 \big( \sum_{\substack{ \theta_i \subseteq B^{(i)}: \\ \theta_i \; R^{-\frac{1}{2}}-ball}} \| f_{i, \theta_i} \|^2_{L^3(w_{B_R})} \big)^{\frac{1}{2}}.
\end{equation}
\end{corollary}

We obtain the following consequence:
\begin{proposition}
Let $N \in 2^{\N_0}$, $\nu \in (0,1]$, $\alpha \in (0,1]$, and $f_i: \T^2_\gamma \to \C$ with $\text{supp}(\hat{f}_i) \subseteq B^{(i)} \subseteq B(0,N)$. We suppose for the area of the triangle $\Delta(\xi_1,\xi_2,\xi_3)$ spanned by $\xi_i \in B^{(i)}$:
\begin{equation*}
| \Delta(\xi_1,\xi_2,\xi_3) | \gtrsim N^2 \nu.
\end{equation*}
Then the following estimate holds:
\begin{equation*}
\int_{[0,N^{-\alpha}] \times \T^2_\gamma} \prod_{i=1}^3 \big| e^{it \Delta} f_i \big| dx dt \lesssim \log(N)^c N^{-\frac{\alpha}{2}} \nu^{-\frac{1}{2}} \prod_{i=1}^3 \| f_i \|_{L^2(\T^2_\gamma)}.
\end{equation*}
\end{proposition}
\begin{proof}
The argument relies on continuous approximation and the decoupling estimate with sharp dependence on transversality from Corollary \ref{cor:DecouplingL3Transversality}. By the rescaling argument and spatial periodicity we obtain like above
\begin{equation*}
\begin{split}
\int_{[0,N^{-\alpha}] \times \T^2_\gamma)} \prod_{i=1}^3 \big| e^{it \Delta} f_i \big| &= N^{-4} \int_{[0,N^{2-\alpha}] \times N \T^2_\gamma} \prod_{i=1}^3 \big| \sum_{k_i \in \Z^2 / N} e^{i(k_i x - t k_i^2)} b_{i,k_i} \big| \\
&\lesssim N^{-4} N^{-2(1-\alpha)} \int_{B_{N^{2-\alpha}}} \prod_{i=1}^3 | \mathcal{E} \tilde{f}_i | dx dt.
\end{split}
\end{equation*}
We include a weight $w_{B_{N^{2-\alpha}}}^3$ and let $\mathcal{E} \tilde{f}_i \cdot w_{B_{N^{2-\alpha}}} = g_i$ to use Corollary \ref{cor:DecouplingL3Transversality} with $R = N^{2-\alpha}$ and $\tilde{\theta}_i$ of size $N^{-(1-\frac{\alpha}{2})} \times N^{-(1-\frac{\alpha}{2})} \times N^{-(2-\alpha)}$, which yields
\begin{equation*}
\int_{B_{N^{2-\alpha}}} \big| \prod_{i=1}^3 g_i \big| \lesssim \log(N)^c \nu^{-\frac{1}{2}} \prod_{i=1}^3 \big( \sum_{\tilde{\theta}_i} \| g_{i, \tilde{\theta}_i} \|^2_{L^3(w_{B_{N^{2-\alpha}}})} \big)^{\frac{1}{2}}.
\end{equation*}
We reverse the continuous approximation like previously to obtain\footnote{Here we gloss over the summation of the weights on the right hand side, which is handled like in previous sections.}
\begin{equation}
\label{eq:L3DecouplingTorus}
\int_{[0,N^{-\alpha}] \times \T^2_\gamma} \prod_{i=1}^3 |e^{it \Delta} f_i | \lesssim \log(N)^c \nu^{-\frac{1}{2}} \prod_{i=1}^3 \big( \sum_{\theta_i: N^{\frac{\alpha}{2}}-\text{ball}} \| e^{it \Delta} f_{i, \theta_i} \|^2_{L^3([0,N^{-\alpha}] \times \T^2_\gamma)} \big)^{\frac{1}{2}}.
\end{equation}
We estimate by H\"older's inequality in time and Bernstein's inequality for $\theta_i$ an $N^{\frac{\alpha}{2}}$-ball:
\begin{equation}
\label{eq:SmallBallBernsteinDecoupling}
\| e^{it \Delta} f_{i, \theta_i} \|_{L^3([0,N^{-\alpha}] \times \T^2_\gamma)} \lesssim N^{-\frac{\alpha}{3}} N^{\alpha \big( \frac{1}{2} - \frac{1}{3} \big)} \| f_{i, \theta_i} \|_{L^2(\T^2_\gamma)}.
\end{equation}
Plugging \eqref{eq:SmallBallBernsteinDecoupling} into \eqref{eq:L3DecouplingTorus} yields the claimed estimate.
\end{proof}
\subsubsection{An extrapolation conjecture}
In the following let $f_i : \R^2 \to \C$ with $\text{supp}(f_i) \subseteq B_2(0,1)$ for $i=1,2,3$ and define
\begin{equation*}
\text{Dec}_{p,R}(f) = \big( \sum_{\theta_i: R^{-\frac{1}{2}}-\text{ball}} \| \mathcal{E} f_{\theta_i} \|^2_{L^p(w_{B_R})} \big)^{\frac{1}{2}}.
\end{equation*}

We observe the following inequality from Plancherel's theorem and H\"older's inequality:
\begin{equation}
\label{eq:TrilinearDecouplingL2}
\big( \int_{B_3(0,R)} \prod_{i=1}^3 \big| \mathcal{E} f_i \big|^{\frac{2}{3}} \big)^{\frac{1}{2}} \lesssim \big( \prod_{i=1}^3  \big( \text{Dec}_{2,R}(f_i) \big)^{\frac{1}{3}}.
\end{equation}
In the previous paragraph we have seen that under the additional transversality assumption
\begin{equation}
\label{eq:TransversalityExtrapolation}
|\mathfrak{n}(\xi_1) \wedge \mathfrak{n}(\xi_2) \wedge \mathfrak{n}(\xi_3)| \gtrsim \nu
\end{equation}
for $\xi_i \in \text{supp}(f_i)$, the following $L^3$-decoupling estimate holds:
\begin{equation}
\label{eq:TrilinearDecouplingL3}
\big\| \prod_{i=1}^3 \big| \mathcal{E} f_i \big|^{\frac{1}{3}} \|_{L^3(B(0,R))} \lesssim \log(R)^c \nu^{-\frac{1}{6}} \big( \prod_{i=1}^3 \text{Dec}_{3,R}(f_i) \big)^{\frac{1}{3}}.
\end{equation}
This leads us to conjecture that under the transversality assumption \eqref{eq:TransversalityExtrapolation} it holds 
\begin{equation}
\label{eq:TrilinearDecouplingL4Transversality}
\big( \int_{B_R} \prod_{i=1}^3 \big| \mathcal{E} f_i \big|^{\frac{4}{3}} \big)^{\frac{1}{4}} \lesssim \log(R)^c \nu^{-\frac{1}{4}} \big( \prod_{i=1}^3 \text{Dec}_{4,R}(f_i) \big)^{\frac{1}{3}}.
\end{equation}

\section{Refined trilinear Strichartz estimates on rescaled tori}
\label{section:StrichartzEstimatesLargeTorus}

In this section we translate the previously proved Strichartz estimates on the unit torus for frequency-dependent times to the large $\lambda$-torus on the unit time scale. The latter is the common setting for the $I$-method.

\subsection{Preliminaries}

We use the following conventions: Let $\alpha_1 = 1$ and $\alpha_i \in (1/2,1]$ for $i=2,\ldots,d$. We endow the (possibly irrational) torus
\begin{equation*}
\T^d_\lambda := (\lambda \T)^d = \prod_{i=1}^d (\R / (2\pi \lambda \alpha_i \Z))
\end{equation*}
with Lebesgue measure, and we write
\begin{equation*}
\| f \|_{L^p(\T^d_\lambda)} = \| f \|_{L^p_\lambda} = \big( \int_{\T^d_\lambda} |f(x)|^p dx \big)^{\frac{1}{p}}
\end{equation*}
for $1 \leq p <\infty$ with the usual modification for $p= \infty$.
Note that we do not distinguish between rational and irrational tori to lighten the notation and since the following estimates are independent of rationality.

Denote $\Z_\lambda = \Z/\lambda$. The Fourier coefficients of $f \in L^1(\T^d_\lambda)$ are given by
\begin{equation*}
\hat{f}(k) = \int_{\T^d_\lambda} e^{- i k x} f(x) dx, \quad k \in (\Z_\lambda / \alpha_1) \times (\Z_\lambda / \alpha_2) \ldots \times (\Z_\lambda / \alpha_d) =: \Z_{\lambda}^d.
\end{equation*}
To be consistent with the notation in physical space, we suppress the $\alpha$-dependence in frequency space. The Fourier inversion formula is given by
\begin{equation*}
f(x) = \lambda^{-d} \sum_{k \in \Z^d_{\lambda}} e^{i k x} \hat{f}(k).
\end{equation*}
The prefactor normalizes the counting measure such that for $\lambda \to \infty$ the normalized counting measure weakly converges to Lebesgue measure. Moreover, Plancherel's theorem holds:
\begin{equation*}
\| f \|^2_{L^2_\lambda} = \frac{1}{\lambda^d} \sum_{k \in \Z_\lambda^d} |\hat{f}(k)|^2. 
\end{equation*}

We record the following improved linear Strichartz estimates, which follow from the unit-time Strichartz estimates on the unit torus (cf. \cite{Bourgain1993A,BourgainDemeter2015}) and the short-time Strichartz estimates without loss (cf. \cite{BurqGerardTzvetkov2004}) after rescaling:

\begin{proposition}
Let $\lambda \geq 1$, $d \in \{1,2\}$, $N \in 2^{\N}$, and $p=\frac{2(d+2)}{d}$. The following estimate holds:
\begin{equation*}
\| P_N U_\lambda(t) f \|_{L^p_{t,x}([0,1] \times \T^d_\lambda)} \lesssim L_d(\lambda,N) \| f \|_{L^2_\lambda}
\end{equation*}
with
\begin{equation}
\label{eq:RescaledLinearStrichartzConstant}
 L_d(\lambda,N) = 1 \text{ for } \lambda \geq N \text{, and } L_d(\lambda,N) = 
\begin{cases}
C_\varepsilon \log(N)^{2+\varepsilon}, &d=1 \\
C_\varepsilon N^\varepsilon, &d=2
\end{cases}
\text{ for } N \geq \lambda.
\end{equation}
\end{proposition}
\begin{proof}
We write
\begin{equation}
\label{eq:RescaledLinearStrichartz1}
\begin{split}
\| P_N U_\lambda(t) f \|^p_{L^p_{t,x}([0,1] \times \T^d_\lambda)} &= \big\| \big( \frac{1}{\lambda} \big)^d \sum_{\substack{|k| \sim N, \\ k \in \Z^d_\lambda}} e^{i (kx - tk^2)} a_k \big\|^p_{L^p_{t,x}([0,1] \times \T^d_\lambda)} \\
&= \lambda^{-d p} \big\| \sum_{\substack{|k'| \sim N \lambda, \\ k' \in \Z}} e^{i \big( \frac{k' x}{\lambda} - \frac{t (k')^2}{\lambda^2} \big)} a_{k'/\lambda} \big\|^p_{L^p_{t,x}([0,1]\times \T^d_\lambda)}.
\end{split}
\end{equation}
By change of variables $x' = x/\lambda$, $t' = t/\lambda^2$ we find
\begin{equation*}
\eqref{eq:RescaledLinearStrichartz1} = \lambda^{-dp} \lambda^{2+d} \big\| \sum_{\substack{|k'| \sim N \lambda, \\ k' \in \Z}} e^{i(k' x' - t' (k')^2)} a_{k'/\lambda} \big\|^p_{L^p_{t,x}([0,\lambda^{-2}] \times \T^d)}.
\end{equation*}
If $\lambda \geq N$, we can use the short-time Strichartz estimates due to Burq--G\'erard--Tzvetkov \cite{BurqGerardTzvetkov2004} without loss to find
\begin{equation*}
\lesssim \lambda^{-dp} \lambda^{2+d} \big( \sum_{\substack{|k'| \sim N \lambda, \\ k' \in \Z}} a_{k'/\lambda}^2 \big)^{\frac{p}{2}}.
\end{equation*}
This implies the claim taking $p$th root.

If $\lambda \leq N$, we can use the finite-time Strichartz estimates on the torus to obtain in case $d=1$:
\begin{equation*}
\begin{split}
&\quad \lambda^{-6} \lambda^3 \big\| \sum_{\substack{|k'| \sim N \lambda, \\ k' \in \Z }} e^{i (k' x' - t'(k')^2)} a_{k'} \big\|^6_{L^6_{t,x}([0,\lambda^{-2}] \times \T^d)} \\
&\lesssim \lambda^{-3} \big( \log(N \lambda)^{2+\varepsilon} \big)^6 \big( \sum_{\substack{|k'| \sim N \lambda, \\ k' \in \Z }} |a_{k'}|^2 \big)^3.
\end{split}
\end{equation*}

If $\lambda \leq N$ and $d=2$, we find
\begin{equation*}
\begin{split}
&\quad \lambda^{-8} \lambda^4 \big\| \sum_{\substack{|k'| \sim N \lambda, \\ k' \in \Z}} e^{i(k' x' - t'(k')^2)} a_{k'} \big\|^4_{L^4_{t,x}([0,\lambda^{-2}] \times \T^2)} \\
&\lesssim (\lambda N)^\varepsilon \| f \|_{L^2_\lambda}^4.
\end{split}
\end{equation*}

\end{proof}

In one dimension, we shall also use the following rescaled bilinear Strichartz estimate, which is \cite[Proposition~3.7]{DeSilvaPavlovicStaffilaniTzirakis2007}:
\begin{proposition}
\label{prop:RescaledBilinearStrichartz}
Let $\eta \in C^\infty_c([-2,2])$ and $\eta \equiv 1$ on $[-1,1]$. Let $f_i: \T^d_\lambda \to \C$ be $\lambda$-periodic functions whose Fourier transforms are supported on $\{ k: |k| \sim N_1 \}$ and $\{k: |k| \sim N_2 \}$ respectively with $N_1 \gg N_2$. Then
\begin{equation}
\label{eq:RescaledBilinearStrichartz1d}
\| \eta(t) U_\lambda(t) f_1 \, U_\lambda(t) f_2 \|_{L_t^2 L_x^2} \lesssim B_1(\lambda,N_1) \| f_1 \|_{L^2} \| f_2 \|_{L^2},
\end{equation}
where
\begin{equation}
\label{eq:RescaledBilinearConstant1d} 
B_1(\lambda,N_1) = 
\begin{cases}
1, \qquad &\text{if } N_1 \leq 1, \\
\big( \frac{1}{N_1} + \frac{1}{\lambda} \big)^{\frac{1}{2}}, \qquad &\text{if } N_1 > 1.
\end{cases}
\end{equation}
\end{proposition}
The constant in \eqref{eq:RescaledBilinearStrichartz1d} reflects that as $\lambda \geq N_1$ the periodic propagation behaves essentially like on Euclidean space,  and we recover the refined bilinear Strichartz estimates due to Bourgain \cite{Bourgain1998}.

In two dimensions, we have the corresponding result by Fan \emph{et al.} \cite{FanStaffilaniWangWilson2018}:
\begin{theorem}
Let $f_i \in L^2(\T^2_\lambda)$, $i=1,2$ with $\text{supp}(\hat{f}_i) \subseteq \{ k: |k| \sim N_i \}$, $i=1,2$, for some $N_1 \geq N_2 \gg 1$.
Then it holds:
\begin{equation*}
\| \eta(t) U_\lambda(t) f_1 \, U_\lambda(t) f_2 \|_{L^2_{t,x}} \lesssim_\varepsilon N_2^\varepsilon B_2(\lambda,N_1,N_2) \| f_1 \|_2 \| f_2 \|_2
\end{equation*}
with
\begin{equation*}
B_2(\lambda,N_1,N_2) = \big( \frac{1}{\lambda} + \frac{N_2}{N_1} \big)^{\frac{1}{2}}.
\end{equation*}
\end{theorem}

On the square torus the result was already proved by De Silva \emph{et al.} \cite[Proposition~4.6]{DeSilvaPavlovicStaffilaniTzirakis2007}; see also Burq--G\'erard--Tzvetkov \cite{BurqGerardTzvetkov2005}. We record the following consequence by Galilean invariance:
\begin{corollary}
Under the above assumptions, if additionally $N_3 \leq N_2 \ll N_1$, and $\text{supp}(\hat{f}_2) \subseteq \{ k: |k| \sim N_2 \} \cap B(\xi_2^*,N_3)$ for some $\xi_2^* \in B(0,N_2)$, we have
\begin{equation*}
\| \eta(t) \, U_\lambda(t) f_1 \, U_\lambda(t) f_2 \|_{L^2_{t,x}} \lesssim_\varepsilon N_3^\varepsilon B_2(\lambda,N_1,N_3) \| f_1 \|_2 \| f_2 \|_2.
\end{equation*}
\end{corollary}

\subsection{The one-dimensional case}

With linear and bilinear Strichartz estimates on rescaled tori at hand, we show the following, from which trilinear refinements follow:
\begin{proposition}
Let $\lambda \geq 1$, $N_i \in 2^{\N}, \, i=1,2$, $N_1 \gg N_2$ and $N_1 \geq N_2^{1+\beta}$ for some $\beta \in (0,1]$. Suppose that $\lambda \leq N_1$ and $\text{supp} (\hat{f}_i) \subseteq \{ \xi \in \R : |\xi| \sim N_i \}$. Then we have
\begin{equation}
\label{eq:TrilinearImprovedStrichartz}
\int_{[0,1] \times \T_{\lambda}} | U_\lambda(t) f_1|^2 | U_\lambda(t) f_2 |^4 dt dx \lesssim  B_1(\lambda,N_1)^{\frac{\beta}{1+\beta}} L_1(\lambda,N_1) \| f_1 \|_2^2 \| f_2 \|_2^4
\end{equation}
with $B_1$ defined in \eqref{eq:RescaledBilinearConstant1d} and $L_1$ defined in \eqref{eq:RescaledLinearStrichartzConstant}.
\end{proposition}

\begin{remark}
\eqref{eq:TrilinearImprovedStrichartz} still holds for $\lambda \geq N_1$, but is inferior to the estimate obtained from directly using the bilinear Strichartz estimate together with Bernstein's inequality.
Let $N_2 = N_1^{\frac{1}{1+\beta}}$ and $\lambda \geq N_1$. We find by H\"older's inequality
\begin{equation}
\label{eq:DirectEstimate}
\int_{[0,1] \times \T_\lambda} |U_\lambda(t) f_1|^2 |U_\lambda(t) f_2|^4 dx dt \lesssim B_1(\lambda,N_1)^2 N_2 \| f_1 \|_2^2 \| f_2 \|_2^4 \lesssim N_1^{-\frac{\beta}{1+\beta}} \| f_1 \|_2^2 \| f_2 \|_2^4.
\end{equation}
This strengthens \eqref{eq:TrilinearImprovedStrichartz}, which yields
\begin{equation*}
\int_{[0,1] \times \T_\lambda} |U_\lambda(t) f_1|^2 |U_\lambda(t) f_2|^4 dx dt \lesssim N_1^{-\frac{\beta}{2(1+\beta)}} \prod_{i=1}^3 \| f_1 \|_2^2 \| f_2 \|_2^4.
\end{equation*}
Note that this also shows that the trilinear refined estimate does not give an improvement on Euclidean space.

On the other hand, if $N_2 = N_1^{\frac{1}{1+\beta}}$ and $\lambda \leq N_1$, \eqref{eq:DirectEstimate} becomes
\begin{equation}
\label{eq:DirectEstimate2}
\int_{[0,1] \times \T_\lambda} |U_\lambda(t) f_1|^2 |U_\lambda(t) f_2|^4 dx dt \lesssim \frac{N_2}{\lambda} \| f_1 \|_2^2 \| f_2 \|_2^4.
\end{equation}
Compare this to \eqref{eq:TrilinearImprovedStrichartz}, which yields
\begin{equation}
\label{eq:TrilinearEstimate2}
\int_{[0,1] \times \T_\lambda} |U_\lambda(t) f_1|^2 |U_\lambda(t) f_2|^4 dx dt \lesssim \lambda^{-\frac{\beta}{2(1+\beta)}} \log(N_1)^c \| f_1 \|_2^2 \| f_2 \|_2^4.
\end{equation}
\eqref{eq:TrilinearEstimate2} improves on \eqref{eq:DirectEstimate2} once
\begin{equation*}
N_1 \log(N_1)^c \lesssim \lambda^{\frac{1}{2}}.
\end{equation*}
\end{remark}

\begin{proof}
The proof is essentially a reprise of Sections \ref{section:Preliminaries} and \ref{section:ProofTheoremB}, and we shall be brief. By almost orthogonality, we can suppose that $f_1$ is Fourier supported in an interval of length $N_2$ and $\| f_j \|_2 = 1$ for $j=1,2$. 

We use the continuous approximation and rescale to $N_1^2 \times \lambda N_1$. This yields a scaling factor of $N_1^{-3}$. Via $\lambda$-periodicity the $x$-domain is extended to $N_1^2$, and we find:
\begin{equation*}
\int_{[0,1] \times \T_\lambda} |U_\lambda(t) f_1|^2 |U_\lambda(t) f_2|^4 dx dt \lesssim N_1^{-4} \lambda \int_{B_{N_1^2}} |\mathcal{E} \tilde{f}_1|^2 |\mathcal{E} \tilde{f}_2|^4 dx dt.
\end{equation*}
We dominate $1_{B_{N_1^2}} \lesssim w_{B_{N_1^2}}^6$ with a suitable Schwartz function, which is compactly supported in Fourier space to find
\begin{equation*}
\lesssim N_1^{-4} \lambda \int_{\R^2} |g_{I_1}|^2 |g_{I_2}|^4 dx dt
\end{equation*}
with notations like above: $I_j$ denote $1$-separated intervals in $[0,1]$ with length less than $\tilde{\beta} =N_2/N_1 = N_1^{-\frac{\beta}{1+\beta}}$. Moreover, the functions $g_{I_j}$ have Fourier support in the $N_1^{-2}$-neighbourhood of the paraboloid $\{ (\xi,-\xi^2) : \xi \in I_j \}$. We invoke the decoupling iteration (after changing notations) to find:

\begin{equation}
\label{eq:DecouplingIterationRescaled}
\begin{split}
&\quad \int |g_{I_1}|^2 |g_{I_2}|^4 \\
 &\leq C_1 \| g_{I_2} \|_6^3 \sum_{|I_1^2| = \tilde{\beta}^2} \big( \int \big( |g_{I^2_1}|^4 * \phi_{I_1} \big) \big( |g_{I_2}|^2 * \phi_{I_2} \big) \big)^{\frac{1}{2}} \\
&\leq C_1^{\frac{3}{2}} \| g_{I_2} \|^3_6 \sum_{|I_1^2| = \tilde{\beta}^2} \big( \sum_{|I_2^2| = \tilde{\beta}^4} \int \big( |g_{I_1^2}|^4 * \phi_{I_1^2} \big) \big( |g_{I_2^2}|^2 * \phi_{I_2^2} \big) \big)^{\frac{1}{2}} \\
&\; \vdots \\
&\leq C_1^2 \| g_{I_2} \|^3_{L^6} \sum_{|I_1^2| = \tilde{\beta}^2} \| g_{I_1^2} \|_{L^6}^{\frac{3}{2}} \\
&\quad \times \big( \sum_{|I_2^2| = \tilde{\beta}^4} \| g_{I_2^2} \|_{L^6}^{\frac{3}{2}} \big( \sum_{|I_1^3| = \tilde{\beta}^4} \| g_{I_1^3} \|_{L^6}^{3/2} \ldots \big( \int |g_{I_1^{m/2}}|^2 * \phi_{I_1^{m/2}} |g_{I_2^{m/2}}|^4 * \phi_{I_2^{m/2}} \big)^{\frac{1}{2}} \big)^{\frac{1}{2}} \ldots \big)^{\frac{1}{2}}.
\end{split}
\end{equation}
The number of decoupling iterations is estimated by $m = \lceil \log_2(\beta^{-1} + 1) \rceil$ (which we suppose to be even to simplifiy notations) because
\begin{equation*}
N_1^{-\frac{\beta}{1+\beta}2^m} \leq N_1^{-\frac{\beta}{1+\beta} (1+\beta^{-1})} = N_1^{-1}.
\end{equation*}

We estimate the linear and bilinear contributions like in Section \ref{section:ProofTheoremB} by reversing the continuous approximation and using the Strichartz estimates on large tori:
For the linear part we have
\begin{equation}
\label{eq:LinearExitRescaled}
\| g_I \|_{L^6(\R^2)} \lesssim \big( \frac{N_1^4}{\lambda} \big)^{\frac{1}{6}} \| U_\lambda(t) \tilde{f}_I \|_{L_t^6([0,1], L^6_x(\T_\lambda))} \lesssim  \big( \frac{N_1^4}{\lambda} \big)^{\frac{1}{6}} L_1(\lambda,N_1) \| b_I \|_{\ell^2_\lambda}.
\end{equation}

Like in Section \ref{section:ProofTheoremB}, regarding the expression obtained after decoupling we abandon all multilinearity to find\footnote{Here we gloss over the technical issue of estimating the convolution with $L^1$-normalized bump function, which can be handled like in Section \ref{section:ProofTheoremB}.}
\begin{equation*}
\int |g_{I_1^{m/2}}|^2 * \phi_{I_1^{m/2}} |g_{I_2^{m/2}}|^4 * \phi_{I_2^{m/2}} \lesssim \int |g_{I_1^{m/2}}|^2 |g_{I_2^{m/2}}|^2 dt dx \; \| g_{I_2^{m/2}} \|_{L^\infty_{t,x}}^2
\end{equation*}
and now scale back, use periodicity, and the estimate from Proposition \ref{prop:RescaledBilinearStrichartz} together with Bernstein's inequality. Note that Bernstein's inequality does not lose derivatives in the present context. We find
\begin{equation}
\label{eq:BilinearExitRescaled}
\int |g_{I_1^{m/2}}|^2 |g_{I_2^{m/2}}|^2 dt dx \; \| g_{I_2^{m/2}} \|_{L^\infty_{t,x}}^2 \lesssim \big( \frac{N_1^4}{\lambda} \big) B_1^2(\lambda,N_1) \| b_{I_1^{m/2}} \|_{\ell^2_\lambda}^2 \| b_{I_2^{m/2}} \|_{\ell^2_\lambda}^4.
\end{equation}
Plugging \eqref{eq:LinearExitRescaled} and \eqref{eq:BilinearExitRescaled} into \eqref{eq:DecouplingIterationRescaled}, we find
\begin{equation*}
\int |g_{I_1}|^2 |g_{I_2}|^4 dt dx \leq C_1^2 \big( \frac{N_1^4}{\lambda} \big) B_1(\lambda,N_1)^{1/2^{m-1}} \| b_{I_1} \|_{\ell^2_\lambda}^2 \| b_{I_2} \|^4_{\ell^2_\lambda}.
\end{equation*}
Since $B_1(\lambda,N_1)^{1/2^{m-1}} \leq B_1(\lambda,N_1)^{\frac{\beta}{1+\beta}}$, this proves \eqref{eq:TrilinearImprovedStrichartz} in case $N_1 \geq \lambda$.
\end{proof} 

\subsection{The two-dimensional case}

In this section we show the following variant of Theorem \ref{thm:TrilinearStrichartzEstimates2d}:
\begin{theorem}
\label{thm:RescaledEstimate2d}
Let $\lambda \geq 1$, $N_i \in 2^{\N}$, $i=1,2,3$, $N_1 \geq N_2$, $N_1 \geq N_3$, and $N_1 \geq N_3^{1+\beta}$ for some $\beta \in (0,1]$. Suppose that $\lambda \leq N_1$, $\text{supp} (\hat{f}_i) \subseteq \{ \xi \in \R : |\xi| \sim N_1 \}$, and $\text{supp} (\hat{f}_i) \subseteq B(\xi_i^*,N_3)$ for some $\xi_i^* \in \R^2$ with $|\xi_i^*| \leq 2N_1$. Suppose that $\xi_i \in \text{supp}(\hat{f}_i)$ satisfy the transversality assumption \eqref{eq:Transversality3d}.

Then the following estimate holds:
\begin{equation}
\label{eq:TrilinearImprovedStrichartz2d}
\int_{[0,1] \times \T_{\lambda}} \prod_{i=1}^3 | U_\lambda(t) f_i|^{\frac{4}{3}} dx dt \lesssim_\varepsilon N_1^\varepsilon \nu^{-1+\frac{\beta}{2(1+\beta)}} B_2(\lambda,N_1,1)^{\frac{2 \beta}{3(1+\beta)}}  \big( \prod_{i=1}^3 \| f_i \|_{L^2_\lambda} \big)^{\frac{4}{3}}
\end{equation}
with $B_2$ defined in \eqref{eq:RescaledBilinearConstant1d}.
\end{theorem}
\begin{proof}
The argument follows Section \ref{section:TrilinearStrichartz2d} closely. We use rescaling and continuous approximation to find
\begin{equation*}
\begin{split}
\int_{[0,1] \times \T^2_\lambda} \big| U_\lambda(t) f_1 U_\lambda(t) f_2 U_\lambda(t) f_3 \big|^{\frac{4}{3}} dx dt &\lesssim N_1^4 \int_{[0,N_1^2] \times N_1 \T^2_\lambda} \big| U_\lambda(t) \tilde{f}_1 U_\lambda(t) \tilde{f}_2 U_\lambda(t) \tilde{f}_3 \big|^{\frac{4}{3}} dx dt \\
&\lesssim N_1^4 (N_1 \lambda^{-1})^2 \int_{\R^3} \big| \mathcal{E} \tilde{f}_1 \mathcal{E} \tilde{f}_2 \mathcal{E} \tilde{f}_3 \big|^{\frac{4}{3}} w_{B_{N_1^2}}^{\frac{4}{3}} dx dt.
\end{split}
\end{equation*}
Above $w_{B_{N_1^2}}$ denotes a weight rapidly decaying off $B_{N_1^2}$, which is compactly supported in Fourier space. We turn to decoupling of
\begin{equation*}
\frac{1}{|B_{N_1}|^2} \int_{\R^3} |\tilde{F}_1 \tilde{F}_2 \tilde{F}_3|^{\frac{4}{3}} dx dt,
\end{equation*}
where $\tilde{F}_i$ denotes functions, which are supported in the $N_1^{-2}$-neighbourhood of the truncated paraboloid, their spatial Fourier support is confined to balls of radius $N_1^{-\frac{\beta}{1+\beta}}$ and $\xi_i \in \pi_{\R^2}(\text{supp} \hat{F}_i)$ satisfy
\begin{equation*}
| (2\xi_1,1) \wedge (2\xi_2,1) \wedge (2\xi_3,1)| \gtrsim \nu.
\end{equation*}
We use notations $A_p$ and $D_p$ from Section \ref{subsection:DecouplingIteration2d}. Moreover, we have
\begin{equation*}
2^{- \underline{m}} = \frac{N_3}{N_1} \leq N_1^{-\frac{\beta}{1+\beta}} \text{ and } 2^{\overline{m}} = N_1^{2} = \delta^{-2},
\end{equation*}
and the multiscale inequality becomes
\begin{equation}
\label{eq:RescaledMultiscaleInequality}
\begin{split}
A_p^p(\underline{m},Q^{\overline{m}},\underline{F},\delta) &\lesssim \nu^{-\frac{1}{2}} \log(N_1)^c \big( \prod_{\ell=1}^{n-1} \nu^{-\frac{\kappa_p^\ell}{2}} \log(N_1)^{c \kappa_p^{\ell}} \big) \\
&\quad \times A_p^p(2^n \underline{m}, Q^{\overline{m}}, \underline{F},\delta)^{\kappa_p^n} \prod_{\ell=1}^{n} D_p^p(2^{\ell -1} \underline{m}, Q^{\overline{m}}, \underline{F},\delta)^{\kappa_p^{\ell - 1} (1-\kappa_p)}.
\end{split}
\end{equation}
Recall that $\kappa_4 = \frac{1}{2}$ and like above, the number of iterations of the multiscale inequality is at most
\begin{equation*}
m = \lceil \log_2(\beta^{-1} + 1) \rceil.
\end{equation*}
Presently, we use linear and bilinear Strichartz estimates on the $\lambda$-torus for finite times to exit the decoupling iteration. We record the analogs of Lemma \ref{lem:LinearExitEstimate2d} and Proposition \ref{prop:EstimateApp}:
\begin{lemma}
\label{lem:RescaledLinearExit2d}
The following estimate holds:
\begin{equation*}
\begin{split}
D_4^4(m,Q,\underline{F},\delta) &= \frac{1}{|Q|} \big[ \prod_{i=1}^3 \big( \sum_{I_{i,m} \in \mathbb{I}_m} \| \mathcal{P}_{I_{i,m}} F_i \|^2_{L^4(w_{Q})} \big)^{\frac{1}{2}} \big]^\frac{4}{3} \\
&\lesssim \frac{N_1^{-6} \lambda^2}{|Q|} L_2^4(\lambda,N_1) \big( \prod_{i=1}^3 \| f_i \|_2 \big)^{\frac{4}{3}}.
\end{split}
\end{equation*}
\end{lemma}
For the multilinear expression we find via exiting the decoupling iteration with bilinear Strichartz estimates:
\begin{proposition}
\label{prop:RescaledBilinearExitEstimate2d}
The following estimate holds:
\begin{equation*}
 A_4^4(2^n \underline{m},Q^{\overline{m}},\underline{F},\delta) \lesssim_\varepsilon \frac{1}{|Q|} N_1^{-6} \lambda^2 B_2^{\frac{4}{3}}(\lambda,N_1) L_2^{\frac{4}{3}}(\lambda,N_1) \big( \prod_{i=1}^3 \| f_i \|_2 \big)^{\frac{4}{3}}.
\end{equation*}
\end{proposition}
\begin{proof}
In the following we use the heuristic argument previously described in \eqref{eq:HeuristicLocallyConstant} relying on the uncertainty principle. This can be made rigid like in the proof of Proposition \ref{prop:EstimateApp}. We have
\begin{equation*}
\begin{split}
&\quad \frac{1}{|\mathcal{Q}_{m,r}|} \sum_{\Delta_m \in \mathcal{Q}_{m,r}(Q^r)} \big( \prod_{i=1}^3 \big( \sum_{I_{i,m} \in \mathbb{I}_m} \| \mathcal{P}_{I_{i,m}} F_i \|^2_{L^2_{\#}(w_{\Delta_m})} \big)^{\frac{1}{2}} \big)^{\frac{4}{3}} \\
&\lesssim \frac{1}{|Q|} \int_Q \prod_{i=1}^3 \big( \sum_I |\mathcal{P}_I F_i|^2 \big)^{\frac{4}{6}} \\
&\lesssim \frac{1}{|Q|} \big( \sum_{I_1,I_3} \int_Q |\mathcal{P}_{I_1} F_1 \mathcal{P}_{I_3} F_3|^2 \big)^{\frac{2}{3}} \big( \int_Q \big( \sum_I |\mathcal{P}_I F_2|^2 \big)^2 \big)^{\frac{1}{3}}.
\end{split}
\end{equation*}
At this point we can reverse the continuous approximation and use rescaled Strichartz estimates to find
\begin{equation*}
\lesssim \frac{1}{|Q|} N_1^6 \lambda^{-2} B_2(\lambda,N_1,1)^{\frac{4}{3}} \big( \prod_{i=1}^3 \| f_i \|_{L^2_\lambda} \big)^{\frac{4}{3}}.
\end{equation*}
\end{proof}

Now we can conclude the proof of Theorem \ref{thm:RescaledEstimate2d} by employing Lemma \ref{lem:RescaledLinearExit2d} and Proposition \ref{prop:RescaledBilinearExitEstimate2d} in \eqref{eq:RescaledMultiscaleInequality}:
\begin{equation*}
A_4^4(\underline{m},Q^{\overline{m}},\underline{F},\delta) \lesssim \nu^{-1} \log(N_1)^{2c} B_2(\lambda,N_1,1)^{\frac{1}{2^m} \cdot \frac{4}{3}} L_2(\lambda,N_1) \big( \prod_{i=1}^3 \| f_i \|_{L^2_\lambda} \big)^{\frac{4}{3}}.
\end{equation*}
With $m= \lceil \log_2(\beta^{-1}+1) \rceil$, we find
\begin{equation*}
\int_{[0,1] \times \T_{\lambda}} \prod_{i=1}^3 | U_\lambda(t) f_i|^{\frac{4}{3}} dx dt \lesssim_\varepsilon N_1^\varepsilon B_2(\lambda,N_1,1)^{\frac{2 \beta}{3(\beta +1)}} \big( \prod_{i=1}^3 \| f_i \|_{L^2_\lambda} \big)^{\frac{4}{3}}.
\end{equation*}
\end{proof}
 
\section{Strichartz estimates on frequency dependent time intervals}
\label{section:LinearSmoothing}

This section is concerned with linear Strichartz estimates on frequency-dependent time intervals:
\begin{equation}
\label{eq:ShorttimeStrichartzSEQ}
\| P_N e^{it \Delta} f \|_{L^p_{t,x}([0,N^{-\alpha}] \times \T^d)} \lesssim \| f \|_{L^2(\T^d)}.
\end{equation}

At the critical exponent $p=\frac{2(d+2)}{d}$ in the limiting case $N^{-\alpha} \to \log(N)^{-1}$ this implies global well-posedness of the mass-critical NLS.

In this section we point out how decoupling implies smoothing for $\alpha > 0$ and $2 \leq p < \frac{2(d+2)}{d}$ and for increased dispersion at the critical exponent.

Although the present argument misses \eqref{eq:ShorttimeStrichartzSEQ}, we obtain structural information about possible counterexamples. We have the following variant of \eqref{eq:ShorttimeStrichartzSEQ}:
\begin{proposition}
\label{prop:ShorttimeStrichartzLowP}
Let $\alpha > 0$, and $2 \leq p < \frac{2(d+2)}{d}$. Then the following estimate holds:
\begin{equation}
\label{eq:ShorttimeStrichartzLowP}
\| e^{it \Delta} P_N f \|_{L^p_{t,x}([0,N^{-\alpha}] \times \T^d)} \lesssim N^{-\kappa} \| f \|_{L^2(\T^d)}
\end{equation}
for $\kappa < \alpha \big( \frac{d}{2} \big( \frac{1}{2} - \frac{1}{p} \big) - \frac{1}{p} \big)$.
\end{proposition}
\begin{proof}
We use scaling $t \to N^2 t$, $x \to N x$ to rescale to unit frequencies:
\begin{equation*}
e^{it \Delta} P_N f  = \sum_{|k| \sim N} e^{i ( kx + tk^2)} a_k= \sum_{|k'| \sim 1} e^{i (N k' \cdot x + t N^2 (k')^2)} a_{N k'}
\end{equation*}
and use periodicity to obtain
\begin{equation*}
\begin{split}
\| e^{it \Delta} P_N f \|^p_{L^p_{t,x}([0,N^{-\alpha}] \times \T^d)} &\to N^{-2} N^{-d} \| e^{it \Delta} P_1 \tilde{f} \|^p_{L^p_{t,x}([0,N^{2-\alpha}] \times (N \T)^d} \\
&\sim N^{-(2+d+(1-\alpha) d)} \| e^{it \Delta} P_1 \tilde{f} \|^p_{L^p_{t,x}(B_{N^{2-\alpha}})}.
\end{split}
\end{equation*}
We use continuous approximation
\begin{equation*}
e^{it \Delta} P_1 \tilde{f} = \sum_{k \in (\Z/N)^d, |k| \sim 1} e^{i (k x + tk^2)} a_k \approx \mathcal{E} \tilde{f} 
\end{equation*}
to obtain:
\begin{equation*}
\| e^{it \Delta} P_1 \tilde{f} \|_{L^p_{t,x}(B_{N^{2-\alpha}})} \lesssim \| \mathcal{E} \tilde{f} \|_{L^p_{t,x}(w_{B_{N^{2-\alpha}}})}.
\end{equation*}
This is amenable to $\ell^2$-decoupling\footnote{For $d=1$ the relevant range is $4<p<6$ because for $p=4$ the Fefferman--Córdoba square function estimate holds without loss.}, which gives
\begin{equation*}
\| \mathcal{E} \tilde{f} \|_{L^p_{t,x}(w_{B_{N^{2-\alpha}}})} \lesssim_\varepsilon N^{\varepsilon} \big( \sum_{\theta: N^{-1+\frac{\alpha}{2}}-\text{ball}} \| \mathcal{E} \tilde{f}_\theta \|^2_{L^p_{t,x}(w_{B_{N^{2-\alpha}}})} \big)^{\frac{1}{2}}.
\end{equation*}
After applying $\ell^2$-decoupling, we can reverse the continuous approximation, the scaling and use periodicity to find
\begin{equation*}
\begin{split}
&\quad \big( \sum_{\theta: N^{-1+\frac{\alpha}{2}}-\text{ball}} \| \mathcal{E} \tilde{f}_\theta \|^2_{L^p_{t,x}(w_{B_{N^{2-\alpha}}})} \big)^{\frac{1}{2}} \\ &\lesssim N^{\frac{2}{p}} N^{\frac{d}{p}} N^{\frac{(1-\alpha) d}{p}} \big( \sum_{\theta: N^{\frac{\alpha}{2}}- \text{ball}} \| e^{it \Delta} P_{\theta} f \|^2_{L^p_{t,x}([0,N^{-\alpha}] \times \T^d)} \big)^{\frac{1}{2}}.
\end{split}
\end{equation*}
A note of clarification here: the weight actually requires us to estimate 
\begin{equation*}
\| e^{it \Delta} P_\theta f \|_{L^p_{t,x}(w([0,N^{-\alpha}] \times \T^d))}.
\end{equation*}
 This can be reduced to the above by periodicity in the spatial variables. In the time variable we simply apply the following argument to
\begin{equation*}
\| e^{it \Delta} P_\theta f \|_{L^p_{t,x}([k N^{-\alpha},(k+1) N^{-\alpha}] \times \T^d)}
\end{equation*}
with $k \in \Z$, which is summable in $k$ by the weight $w$ decaying off $[0,N^{-\alpha}]$ and unitarity of the Schr\"odinger propagator in $L^2(\T^d)$.

Lastly, we use H\"older in time and Bernstein's inequality to find
\begin{equation*}
\| e^{it \Delta} P_\theta f \|_{L^p_{t,x}([0,N^{-\alpha}] \times \T^d)} \lesssim N^{-\frac{\alpha}{p}} N^{\frac{\alpha}{2} \cdot d \big( \frac{1}{2} - \frac{1}{p} \big)} \| P_\theta f \|_{L^2(\T^d)}.
\end{equation*}
The proof of \eqref{eq:ShorttimeStrichartzLowP} is complete by $L^2$-orthogonality of $(P_\theta f)_{\theta}$.
\end{proof}

Clearly, the argument extends to irrational tori. 

In another direction the decoupling argument shows that smoothing estimates hold for increased dispersion on frequency-dependent time intervals with \emph{arbitrary} power $N^{-\alpha}$ at the critical exponent. With a nonlinear application in mind, we formulate the result for the Airy propagator: 
\begin{proposition}
\label{prop:ShorttimeStrichartzAiry}
Let $\alpha \in (0,2]$. The following estimate holds:
\begin{equation}
\label{eq:ShorttimeAiryEstimate}
\| P_N e^{t \partial_x^3} f \|_{L_t^6([0,N^{-\alpha}], L_x^6(\T))} \lesssim N^{-\kappa} \| f \|_{L^2(\T)}
\end{equation}
for
\begin{equation*}
\kappa < \kappa_0 =
\begin{cases}
\frac{\alpha}{6}, \quad &\alpha \in (0,1], \\
\frac{1}{6}, \quad &\alpha \in (1,2].
\end{cases}
\end{equation*}
\end{proposition}

We consider $\alpha < 2$ because for larger $\alpha$ we are within the Euclidean window $T=T(N)=N^{-2}$ and the estimates from Euclidean space are expected to hold. Regarding the time localization $T=T(N)=N^{-2}$ we have due to Dinh \cite{Dinh2017}:
\begin{equation*}
\| P_N e^{t \partial_x^3} f \|_{L^6_t([0,N^{-2}],L^6_x(\T))} \lesssim N^{-\frac{1}{6}} \| f \|_{L^2(\T)}.
\end{equation*}
This recovers the estimate from Euclidean space. Note that the smoothing will further improve for $T(N) = N^{-2+\beta}$, $\beta>0$ as a simple consequence of H\"older's and Bernstein's inequality.

\begin{proof}[Proof~of~Proposition~\ref{prop:ShorttimeStrichartzAiry}]
Suppose that $\alpha \in [1,2)$. From scaling and periodicity we find
\begin{equation*}
\begin{split}
\| P_N e^{t \partial_x^3} f \|^6_{L_t^6([0,N^{-\alpha}],L^6(\T))} &\lesssim N^{-4} \| P_1 e^{t \partial_x^3} \tilde{f} \|^6_{L_t^6([0,N^{3-\alpha}],L^6_x(N \T))} \\
&\lesssim N^{-6+\alpha} \| P_1 e^{t \partial_x^3} \tilde{f} \|^6_{L_t^6([0,N^{3-\alpha}],L^6_x(B_{N^{3-\alpha}}))}.
\end{split}
\end{equation*}
After continuous approximation,
\begin{equation*}
P_1 e^{t \partial_x^3} \tilde{f} \approx \mathcal{E} f, \quad \mathcal{E} f = \int e^{i(x \xi + t \xi^3)} f(\xi) d\xi,
\end{equation*}
the expression is amenable to decoupling on the scale $\delta = N^{-\frac{3}{2}+\frac{\alpha}{2}}$ (such that $\delta^{-2} = N^{3 - \alpha}$). The application of decoupling incurs a factor $\log(N)^c$ due to Guth--Maldague--Wang \cite{GuthMaldagueWang2020} and scaling back and reversing the continuous approximation we obtain
\begin{equation}
\label{eq:SmoothingAiry}
\lesssim \log(N)^c \big( \sum_{\theta: N^{-\frac{1}{2}+\frac{\alpha}{2}}-\text{interval}} \| P_\theta e^{t \partial_x^3} f \|^2_{L_t^6([0,N^{-\alpha}],L^6(\T))} \big)^{\frac{1}{2}}.
\end{equation}
Here we sum over $N^{-\frac{1}{2}+\frac{\alpha}{2}}$-intervals $\theta$. It turns out that after using a KdV-Galilean transform we find with $|A| \sim N$ and $M = N^{\frac{1}{2}+\frac{\alpha}{2}}$:
\begin{equation*}
\begin{split}
\sum_{k=A}^{A+M} e^{i (kx + tk^3)} a_k &= e^{i A x} \sum_{\ell=0}^M e^{i (\ell x + t(A+\ell)^3)} a_{\ell + A} \\
&= e^{it A^3 + i A x} \sum_{\ell= 0 }^M e^{i (\ell (x+3A^2 t) + t(3 A \ell^2 + \ell^3))} a_{\ell + A}. 
\end{split}
\end{equation*}
Hence, it suffices to observe that the dispersion relation $\omega(\ell) = 3 A \ell^2 + \ell^3$ does not cause oscillations for $|t| \lesssim N^{-\alpha}$ and $|\ell| \lesssim M$.
We obtain by Bernstein's inequality and H\"older's inequality in time for $\theta$ an $N^{-\frac{1}{2}+\frac{\alpha}{2}}$-interval:
\begin{equation}
\label{eq:SmoothingAiryII}
\| P_\theta e^{t \partial_x^3} f \|_{L_t^6([0,N^{-\alpha}],L^6(\T))} \lesssim N^{-\frac{\alpha}{6}} N^{\big( -\frac{1}{2}+\frac{\alpha}{2} \big) \big( \frac{1}{2}-\frac{1}{6} \big)} \| P_{\theta} f \|_{L^2} = N^{-\frac{1}{6}} \| P_\theta f \|_{L^2}.
\end{equation}
Then the claim follows from plugging \eqref{eq:SmoothingAiryII} into \eqref{eq:SmoothingAiry} and almost orthogonality of $(P_\theta f)_{\theta}$ in $L^2$.

Secondly, we suppose that $\alpha \in (0,1]$. We obtain by rescaling and periodicity
\begin{equation*}
\begin{split}
\| P_N e^{t \partial_x^3} f \|_{L_t^6([0,N^{-\alpha}],L^6_x(\T))}^6 &\lesssim N^{-5} \| P_1 e^{t \partial_x^3} \tilde{f} \|_{L_t^6([0,N^{3-\alpha}],L^6_x(B_{N^2}))}^6 \\
&= N^{-5} \sum_{I: |I| = N^2} \| P_1 e^{t \partial_x^3} \tilde{f} \|_{L_t^6(I,L^6_x(B_{N^2}))}^6.
\end{split}
\end{equation*}
In the last step we decomposed the time interval into intervals of length $N^2$, which yields $N^{1-\alpha}$ intervals. On every interval we can use the arguments for $\alpha \in [1,2]$, i.e. decoupling, reversing the continuous approximation, and trivial estimate of the resulting exponential sum. This yields 
\begin{equation*}
\| P_1 e^{t \partial_x^3} \tilde{f} \|_{L_t^6(I,L^6_x(B_{N^2}))}^6 \lesssim \log(N)^c N^{4} \| f \|^6_{L^2_x(\T)}.
\end{equation*}
Summing $N^{1-\alpha}$ intervals we obtain
\begin{equation*}
\| P_N e^{t \partial_x^3} f \|_{L_t^6([0,N^{-\alpha}],L^6_x(\T))} \lesssim N^{-\frac{\alpha}{6}} \log(N)^c \| f \|_{L^2(\T)}.
\end{equation*}
The proof is complete.
\end{proof}

\section{Non-existence of solutions to the mKdV equation}
\label{section:ExistenceMKDV}

The modified Korteweg-de Vries equation is given by
\begin{equation}
\label{eq:mKdV}
\left\{ \begin{array}{cl}
\partial_t u + \partial_x^3 u &= \pm u^2 \partial_x u , \quad (t,x) \in \R \times \T, \\
u(0) &= u_0 \in H^s(\T)
\end{array} \right.
\end{equation}
The equation with $+$-sign on the right hand side is referred to as defocusing equation (due to coercivity of the energy), whereas the equation with $-$-sign is referred to as focusing.

The well-posedness theory of the modified Korteweg-de Vries equation is extensively studied (see e.g. \cite{Bourgain1993B,KenigPonceVega1996,CollianderKeelStaffilaniTakaokaTao2003,KappelerTopalov2005,Molinet2012,KappelerMolnar2017,Forlano2022}) and the following brief account is by no means exhaustive. We also refer to the references within the cited works. Unless explicitly stated otherwise, we only consider real-valued solutions in the following.

In this section we argue how the short-time Strichartz estimates for the Airy propagator from the previous section implies non-existence of solutions to the  modified Korteweg-de Vries equation in Sobolev spaces with negative regularity. We give a brief overview of previous results.

\medskip

Kappeler--Topalov \cite{KappelerTopalov2005} (see also \cite{KappelerTopalov2004}) showed that the defocusing (real-valued) modified Korteweg-de Vries equation
is globally well-posed in $L^2(\T)$. In \cite{KappelerTopalov2005} it was exploited that the Miura map is a global fold. The Miura map assigns solutions to the defocusing mKdV in $L^2(\T)$ to solutions to the Korteweg-de Vries equation in $H^{-1}(\T)$. Kappeler-Topalov proved global well-posedness of the (arguable even more prominent) Korteweg-de Vries equation in $H^{-1}(\T)$ in their work \cite{KappelerTopalov2006}.

\smallskip

In \cite{KappelerTopalov2005} solutions in $L^2(\T)$ were defined through extension of the data-to-solution mapping from smooth initial data (see also below). 
This well-posedness result was supplemented by Molinet \cite{Molinet2012}, who proved that the Kappeler-Topalov solutions are weak solutions for $s=0$ and the data-to-solution mapping fails to be weakly continuous as a map from $L^2(\T) \to \mathcal{D}'([0,T] \times \T)$. He also showed the first existence result for solutions to the focusing equation in $L^2(\T)$.

\medskip

 However, the scaling critical regularity is $s=-\frac{1}{2}$ and one might surmise that some probabilistic well-posedness result could still hold in Sobolev spaces of negative regularity. Indeed, Kappeler-Molnar \cite{KappelerMolnar2017} proved local well-posedness of the defocusing equation in Fourier Lebesgue spaces $\mathcal{F} L^p$ for $2<p<\infty$ after renormalization (see \eqref{eq:RenormalizedmKdV}). The arguments rely on complete integrability. This shows that the unrenormalized equation is ill-posed in Fourier Lebesgue spaces of negative Sobolev regularity. Further recent local well-posedness results on the complex-valued mKdV equation in Fourier Lebesgue spaces are due to Chapouto \cite{Chapouto2021,Chapouto2023}. We also mention the recent work of Forlano \cite{Forlano2022}, who unified and simplified previous proofs of global well-posedness for the modified Korteweg-de Vries equation in Sobolev spaces $H^{s}(\mathbb{K})$ with $0 \leq s < \frac{1}{2}$ and $\mathbb{K} \in \{ \mathbb{T}, \mathbb{R} \}$.
 
 \medskip

Conditional upon the conjectured $L^8_{t,x}$-Strichartz estimate
\begin{equation}
\label{eq:L8StrichartzAiry}
\| P_N e^{t \partial_x^3} f \|_{L_{t,x}^8([0,1] \times \T)} \lesssim_\varepsilon N^\varepsilon \| f \|_{L^2(\T)},
\end{equation}
I have proved in \cite{Schippa2020} that there is no data-to-solution mapping to the unrenormalized equation in the following sense in $H^s(\T)$ for $s<0$: A map $S:H^s(\T) \to C([-T,T],H^s(\T))$ where $T=T(\| u_0 \|_{H^s}) > 0$ is referred to as data-to-solution mapping to \eqref{eq:mKdV} if it satisfies the following properties:
\begin{itemize}
\item[(i)] $S(u_0)$ satisfies the equation \eqref{eq:mKdV} in the distributional sense and $S(u_0)(0) = u_0$.
\item[(ii)] There exists a sequence of smooth global solutions $(u_n)$ such that $u_n \to S(u_0)$ in $C([-T,T],H^s)$ as $n \to \infty$.
\end{itemize}
The argument in \cite{Schippa2020} does not rely on complete integrability, but frequency-dependent time localization, and also covers the focusing case. It is conceivable that this approach extends with some modifications to the complex case.

\medskip

Like in \cite{KappelerMolnar2017}, the proof of the non-existence of solutions to \eqref{eq:mKdV} relies on the existence of solutions to the renormalized mKdV equation:
\begin{equation}
\label{eq:RenormalizedmKdV}
\left\{ \begin{array}{cl}
\partial_t u + \partial_x^3 u &= \pm \mathfrak{N}(u), \quad (t,x) \in \R \times \T, \\
u(0) &= u_0 \in H^s(\T)
\end{array} \right.
\end{equation}
with
\begin{equation*}
\mathfrak{N}(u) \widehat (n) = i n |\hat{u}(n)|^2 \hat{u}(n) + in \sum_{\substack{n= n_1 + n_2 + n_3, \\
(n_1+n_2)(n_1+n_3)(n_2+n_3) \neq 0}} \hat{u}(n_1) \hat{u}(n_2) \hat{u}(n_3).
\end{equation*}

\medskip

A consequence of the conjectured $L^8_{t,x}$-Strichartz estimate is the short-time estimate
\begin{equation*}
\| P_N e^{t \partial_x^3} u_0 \|_{L_t^6([0,N^{-\alpha}], L^6(\T))} \lesssim_\varepsilon N^{-\frac{\alpha}{18}+\varepsilon} \| u_0 \|_{L^2(\T)}.
\end{equation*}
This was the key ingredient to show the non-existence for solutions to \eqref{eq:mKdV} at negative Sobolev regularity in \cite{Schippa2020}. The specific choice for the frequency-dependent analysis in \cite{Schippa2020} was actually $\alpha = 1+$, which gives a smoothing of $N^{-\frac{1}{18}}$. Since by Proposition \ref{prop:ShorttimeStrichartzAiry} we have now the improved smoothing $N^{-\frac{1}{6}+}$, we obtain Theorem \ref{thm:Nonexistence}, which was formulated in \cite{Schippa2020} conditional upon \eqref{eq:L8StrichartzAiry}:

\begin{theorem}[Non-existence~of~solutions~to~mKdV]
\label{thm:Nonexistence}
There is $s'<0$ so that for $s'<s<0$ there exists $T=T(\| u_0 \|_{H^s})$ such that there exists a local solution $u \in C([-T,T],H^s(\T))$ to \eqref{eq:RenormalizedmKdV}, and we find the a priori estimate
\begin{equation*}
\sup_{t \in [0,T]} \| u(t) \|_{H^s} \leq C \| u_0 \|_{H^s}
\end{equation*}
to hold. Furthermore, solutions to \eqref{eq:mKdV} do not exist for $s'<s<0$.
\end{theorem}

\section*{Appendix: Short-time bilinear Strichartz estimates via the Fefferman--C\'ordoba square function estimate}
In the following we give another proof of the short-time bilinear Strichartz estimaet in one dimension. The argument uses the continuous approximation from previous sections and as key ingredient the bilinear Fefferman--C\'ordoba square function estimate (see \cite{Fefferman1973,Cordoba1982} or \cite[p.~40]{Demeter2020} for a recent treatise). The argument shows how what is known as small cap decoupling or square function (in $\ell^2$) imply Strichartz estimates on frequency dependent times.

We recall the following square function estimate:
\begin{theorem}
\label{thm:CordobaFefferman}
Let $g_1,g_2 \in \mathcal{S}(\R^2)$ with $\text{supp}(\hat{g}_i) \subseteq \mathcal{N}_{\delta}(\{ (\xi,|\xi|^2) : \xi \in I_i \})$, where $I_1, \, I_2 \subseteq [0,1]$ denote intervals with $\text{dist}(I_1,I_2) \sim 1$. Then the following estimate holds:
\begin{equation*}
\int_{\R^2} |g_1 g_2|^2 \lesssim \sum_{\theta_i \subseteq I_i: \\ \delta-\text{interval}} \int_{\R^2} |g_{1, \theta_1}|^2 |g_{2, \theta_2}|^2.
\end{equation*}
\end{theorem}

In the following we observe how this kind of estimate, which decomposes functions into blocks smaller than the canonical scale, translates to claims about exponential sums on frequency-dependent times. We recover the short-time bilinear Strichartz estimate due to Moyua--Vega \cite{MoyuaVega2008} and Hani \cite{Hani2012}:
\begin{theorem}
Let $f_i \in L^2(\T)$ with $\text{supp}(\hat{f}_i) \subseteq I_i \subseteq [0,N]$ and $\text{dist}(I_1,I_2) \sim N$. Then the following estimate holds:
\begin{equation*}
\int_{[0,N^{-1}] \times \T} |e^{it \Delta} f_1 e^{it \Delta} f_2|^2 dx dt \lesssim N^{-1} \prod_{i=1}^2 \| f_i \|_{L^2(\T)}^2.
\end{equation*}
\end{theorem}
\begin{proof}
The rescaling and continuous approximation like above allows us to write:
\begin{equation*}
\begin{split}
&\quad \int_{[0,N^{-1}] \times \T} |e^{it \Delta} f_1 e^{it \Delta} f_2|^2 dx dt \\
 &\lesssim N^3 \int_{[0,N] \times N \T} \big| \sum_{\substack{k_1 \in \Z/N,\\ k_1 \in \tilde{I}_1}} e^{i(k_1 x - t k_1^2)} b_{1,k_1} \sum_{\substack{k_2 \in \Z/N_1, \\ k_2 \in \tilde{I}_2}} e^{i(k_2 x - t k_2^2)} b_{2,k_2} \big|^2 dx dt \\
&\lesssim N^3 \int_{\R^2} | \mathcal{E} \tilde{f}_1 \mathcal{E} \tilde{f}_2 |^2 w_{B_N}^4 dx dt.
\end{split}
\end{equation*}
We check that $\mathcal{E} \tilde{f}_i \cdot w_{B_N} = g_i$ satisfies the assumptions of Theorem \ref{thm:CordobaFefferman} with $\delta = N^{-1}$ to find that
\begin{equation*}
\int_{\R^2} |g_1 g_2|^2 \lesssim \sum_{\substack{\theta_i \subseteq \tilde{I}_i: \\ \delta - \text{interval}}} \int_{\R^2} |g_{1, \theta_1} g_{2, \theta_2} |^2.
\end{equation*}
For this expression it is possible to reverse the continuous approximation to find that the exponential sum has been trivialized:
\begin{equation*}
\begin{split}
N^3 \sum_{\substack{ \theta_i \subseteq \tilde{I}_i : \\ \delta - \text{interval}}} \int_{\R^2} |g_{1, \theta_1} g_{2, \theta_2}|^2 &\lesssim \sum_{k_i \in I_i} \int_{[0,N^{-1}] \times \T} | e^{i(k_1 x - t k_1^2)} \hat{f}_1(k_1) |^2 |e^{i (k_2 x - t k_2^2)} \hat{f}_2(k_2) |^2 dx dt \\
&\lesssim N^{-1} \prod_{i=1}^2 \sum_{k_i} |\hat{f}_i(k_i)|^2 = N^{-1} \| f_1 \|_{L^2(\T)}^2 \| f_2 \|_{L^2(\T)}^2.
\end{split}
\end{equation*}
The proof is complete.
\end{proof}

\section*{Acknowledgement}

I would like to thank Sebastian Herr, Javier Ramos and Po-Lam Yung for helpful discussions on this set of problems. Financial support by Korea Institute for Advanced Study, grant No.
MG093901 is gratefully acknowledged.

\end{document}